\documentclass[12pt]{article}

\usepackage[a4paper, total={7in, 10in}]{geometry}

\usepackage{epsfig}
\usepackage{latexsym}
\usepackage{caption}
\usepackage{amsfonts}
\usepackage{amsmath,amsthm,enumerate, mathrsfs}
\usepackage{enumerate}
\usepackage{enumitem}
\usepackage{authblk}

\usepackage{graphicx}
\usepackage{float}
\usepackage{pict2e} 
\usepackage{graphics}
\usepackage{tikz}
\usepackage{cite}

\usepackage{soul,color}
\definecolor{lightblue}{rgb}{.88,.88,1}
\sethlcolor{lightblue}
\definecolor{lightgreen}{RGB}{153,255,153}

\definecolor{brightred}{RGB}{255,80,102}

\definecolor{newgreen}{RGB}{0,255,0}

\soulregister\cite7
\soulregister\ref7
\soulregister\pageref7

\usepackage[colorinlistoftodos,prependcaption,textsize=scriptsize, textwidth=15mm,]{todonotes}

\usepackage{tikz-3dplot} 
\usepackage{subcaption}

\numberwithin{equation}{section}
 
\usepackage{url}
\usepackage{hyperref}

\usepackage{xcolor}

\newcommand{\gl}{\left\lfloor\frac{g}{2}\right\rfloor}
\newcommand{\gc}{\left\lceil\frac{g}{2}\right\rceil}

\newcommand{\p}{\mathcal{P}} 
\newcommand{\SP}{\mathcal{SP}}
\newcommand{\SPC}{{\rm SPC}}

\newcommand{\TT}{{\rm TT}}
\newcommand{\T}{{\rm T}}

\newtheorem{theorem}{Theorem}
\newtheorem{observation}[theorem]{Observation}

\newtheorem{definition}[theorem]{Definition}
\newtheorem{conjecture}[theorem]{Conjecture}
\newtheorem{claim}{Claim}

\newtheorem{lemma}[theorem]{Lemma}
\newtheorem{proposition}[theorem]{Proposition}

\newtheorem{remark}{Remark}

\usepackage[misc,geometry]{ifsym}

\begin{document}

\title{\Large \bf Extended Double Covers and Homomorphism Bounds of Signed Graphs\footnote{\noindent This research was financed by the ANR project HOSIGRA (ANR-17-CE40-0022) and the IFCAM project ``Applications of graph homomorphisms'' (MA/IFCAM/18/39). The research of the first author was partially financed by the French government IDEX-ISITE initiative 16-IDEX-0001 (CAP 20-25). The third author was supported by the project with grant number: NSFC 11871439, and was also supported by Fujian Provincial Department of Science and Technology(2020J01268).}}
  
\author[1,2]{Florent Foucaud}
\author[3]{Reza Naserasr} 
\author[3,4]{Rongxing Xu} 

\affil[1]{\small LIMOS, CNRS UMR 6158, Universit\'e Clermont Auvergne, Aubi\`ere, France. E-mail: florent.foucaud@uca.fr}
\affil[2]{\small Univ. Orléans, INSA Centre Val de Loire, LIFO EA 4022, F-45067 Orléans Cedex 2, France.}
\affil[3]{\small Universit\'{e} de Paris, CNRS, IRIF, F-75006, Paris, France. E-mail: reza@irif.fr}
\affil[4]{\small School of Mathematical Sciences, University of Science and Technology of China, Hefei, Anhui, 230026, China. E-mail:xurongxing@ustc.edu.cn.}

\maketitle

 \begin{abstract}
 A \emph{signed graph} $(G, \sigma)$ is a graph $G$ together with an assignment $\sigma:E(G) \rightarrow \{+,-\}$.
 	The notion of homomorphisms of signed graphs is a relatively new development which allows to strengthen the connection between the theories of minors and colorings of graphs. Following this thread of thoughts, we investigate this connection through the notion of Extended Double Covers of signed graphs, which was recently introduced by Naserasr, Sopena and Zaslavsky. 
 	More precisely,  we say that a signed graph $(B, \pi)$ is planar-complete if any planar signed graph $(G, \sigma)$ which verifies the  conditions of a basic no-homomorphism lemma with respect to $(B,\pi)$  admits a homomorphism to $(B, \pi)$. Our conjecture then is that: if $(B, \pi)$ is a connected signed graph with no positive odd closed walk which is planar-complete, then its Extended Double Cover ${\rm EDC}(B,\pi)$ is also planar-complete. We observe that this conjecture largely extends the Four-Color Theorem and is strongly connected to a number of conjectures in extension of this famous theorem. 
 	
A given (signed) graph $(B,\pi)$ \emph{bounds} a class of (signed) graphs if every (signed) graph in the class admits a homomorphism to $(B,\pi)$.
 	In this work, and in support of our conjecture, we prove it for the subclass of signed $K_4$-minor free graphs. Inspired by this development, we then investigate the problem of finding optimal homomorphism bounds for subclasses of signed $K_4$-minor-free graphs with restrictions on their girth and we present nearly optimal solutions. Our work furthermore leads to the development of weighted signed graphs.
 \end{abstract}
 
 \section{Introduction}
 
 Graphs in this work are simple and connected, unless stated otherwise (when it is not the case we refer to them as multigraphs). For graph terminology, we follow \cite{WestBook2000}. We work on the larger realm of signed graphs, for which we introduce the extended terminology next.

 \subsection{Signed graphs}
 A \emph{signed graph} $(G, \sigma)$ is a graph $G$ together with an assignment $\sigma:E(G) \rightarrow \{+,-\}$. The graph $G$ may be referred to as the \emph{underlying graph} of $(G, \sigma)$ and $\sigma$ is called its \emph{signature}. 
 
 Given a closed walk $W$ of $(G, \sigma)$, whose edges are $e_1,e_2, \ldots, e_l$, in the order in which they are traversed (allowing repetition), the sign of $W$ is defined to be the product $\sigma(W) := \sigma(e_1)\sigma(e_2)\cdots\sigma(e_l)$. The walk $W$ is said to be \emph{positive} or \emph{negative} depending on the value of $\sigma(W)$. Since a cycle of a graph is also a closed walk, we naturally have the definition of \emph{positive cycles} and \emph{negative cycles}. Observe that if a closed walk $W$ which starts at a vertex $v$ consists of a closed walk $W_1$ at $v$ followed by a closed walk $W_2$ at $v$, then $\sigma(W)=\sigma(W_1)\sigma(W_2)$. Thus, the sign of a closed walk can be determined by the sign of cycles it is composed of.
 
 A key notion in the study of signed graphs is \emph{switching}: to switch at a vertex $v$ of a signed graph $(G, \sigma)$ is to multiply the signs of all edges in the cut $({v}, V(G)\setminus v)$ by $-$. To switch at a set $X$ of vertices is to switch at all the vertices of $X$ in any sequence; this is equivalent to multiplying the signs of all edges in the cut $(X, V(G)\setminus X)$ by $-$. Two signatures $\sigma_1$ and $\sigma_2$ on $G$ are said to be \emph{switching equivalent}, or equivalent for short, if one can be obtained from the other one by a switching. The relation ``switching equivalent'' is an equivalence relation on the set of all possible signatures of a given graph.

 One of the first theorems in the theory of signed graphs is that the set of negative cycles (equivalently the set of positive cycles) uniquely determines the equivalence class of signatures \cite{Zas82}. More precisely:
 
 \begin{theorem}[\cite{Zas82}]
 	\label{signs-equal}
 	Two signatures $\sigma_1$ and $\sigma_2$ on a graph $G$ are equivalent if and only if they induce the same set of negative cycles.
 \end{theorem}
 
 \subsection{Walk girths and special subclasses}
 \label{section:walk girth and special class}
 Considering sign and parity of the length of a closed walk, we have essentially four different types of closed walks. Following~\cite{NSZ20}, we use the elements of $\mathbb{Z}_2^2$ to denote these types in such a way that if two closed walks $W_1$ and $W_2$ of type $ij$ and $i'j'$ have a common starting point, then the walk $W_1W_2$ is of type $ij+i'j'$ where $+$ is the additive operation of $\mathbb{Z}_2^2$. Thus, a positive even walk is of type 00, a negative even walk is of type 10, a positive odd walk is of type 01 and a negative odd walk is of type 11. The \emph{$ij$-walk girth} of $(G, \sigma)$, denoted $g_{ij}(G, \sigma)$, is the length of a shortest closed walk of type $ij$ (setting it to be $\infty$ when there exists no walk of type $ij$). It is easy to observe (see \cite{NSZ20}) that if $g_{ij}(G, \sigma)=\infty$ for one choice of $ij$, then $g_{kl}(G, \sigma)=\infty$ for some other $kl\in \mathbb{Z}_2^2$, distinct from $ij$. Furthermore, observing that $g_{00}(G, \sigma)=2$ unless $G$ has no edge, we have three special subclasses: $\mathcal{G}_{ij}$, for $ij\in \mathbb{Z}_2^2$ and $ij\neq 00$, is the class of all signed graphs in which every cycle is either of type $00$ or of type $ij$. Of these three classes, $\mathcal{G}_{10}$ and $\mathcal{G}_{11}$ are of special importance to us. They are refered to as \emph{consistent} signed graphs in \cite{NRS13} and together, these two classes contain all signed graphs with no positive odd closed walk. Conversely, assuming $G$ is connected, if $(G, \sigma)$ has no positive odd closed walk, then it must be a signed graph in one of these two classes. As graphs in this work are always connected, this property will be our reference point for the union of these two families of signed graphs. Observe that $\mathcal{G}_{10}$ is the class of all \emph{signed bipartite graphs} and $\mathcal{G}_{11}$ is the class of all signed graphs $(G, \sigma)$ which can be switched to $(G, -)$ (it is known as the class of \emph{antibalanced signed graphs}, see \cite{NSZ20}).

 \subsection{Homomorphisms of signed graphs}
 
 We mentioned that the signs of closed walks of $(G, \sigma)$ determine the switching equivalent class of signed graphs on $G$ to which $(G, \sigma)$ belongs to.
 Thus, for signed graphs equipped with the switching operation, the signs of closed walks uniquely determine the signed graph. 
 %Thus, for signed graphs equipped with the switching operation, signs of closed walks are among deterministic characters.
 Since homomorphisms preserve the fundamental structure of the considered objects, a natural definition of a homomorphism of signed graph $(G, \sigma)$ to signed graph $(H, \pi)$ is given below. We note that this definition works the same for signed graphs with loops and multiedges, and that parallel edges of a same sign do not influence the existence of homomorphisms.
 
 \begin{definition}[\cite{NSZ20}]
 	A \emph{homomorphism} of a signed graph $(G, \sigma)$ to a signed graph $(H, \pi)$ is a mapping $f$ which maps the vertices and edges of $G$ to the vertices and edges of $H$, respectively, with the property that incidences, adjacencies and signs of closed walks are preserved. The mapping $f$ is furthermore a \emph{sign-preserving homomorphism} of $(G, \sigma)$ to $(H, \pi)$ if it preserves the signs of edges. The existence of a homomorphsm is denoted by $(G, \sigma)\to (H, \pi)$.
  \end{definition}
 
 This definition immediately gives us a key ingredient of our work, which is our first ``no-homomorphism lemma'' in the context of signed graphs (this term refers to a necessary condition for a specific homomorphism to exist, see~\cite{Hahn-Tardif} for this concept on usual graphs).
 Observe that the image of a closed walk is a closed walk of same length. As we ask the sign of closed walks to be preserved, when there exists a homomorphism from $(G, \sigma)$ to $(H, \pi)$, then for each $ij\in \mathbb{Z}_2^2$, the length of a shortest walk of type $ij$ in $(G, \sigma)$ is at least as that of $(H, \pi)$. This simple fact, which is a key component of our work, is stated as the following no-homomorphism lemma:  
 
 \begin{lemma}\label{lem:No-Hom-Lemma}[The no-homomorphism lemma]
 	Given signed graphs $(G, \sigma)$ and $(H, \pi)$, if $(G, \sigma) \to (H, \pi)$, then for each $ij\in \mathbb{Z}_2^2$, we have $g_{ij}(G, \sigma)\geq g_{ij}(H, \pi)$. 
 \end{lemma}
 
 Observe that when there are no parallel edges, the vertex-mapping and the condition of preserving adjacencies uniquely determine the edge-mapping, and thus we may define a homomorphism solely by a vertex-mapping, as was originally done in~\cite{NRS15}.

 The notion of homomorphisms is closely related to the notion of sign-preserving homomorphisms. We mention some connection next, but for a proof and further connections we refer to \cite{NSZ20}.
 
 \begin{theorem}[\cite{NSZ20}]\label{thm:SignPreservHom-Hom}
 	Given signed graphs $(G, \sigma)$ and $(H, \pi)$, there exists a homomorphism of $(G, \sigma)$ to $(H, \pi)$ if and only if there exists a signature $\sigma'$, switching equivalent to $\sigma$, and a sign-preserving homomorphism of $(G, \sigma')$ to $(H, \pi)$.
 \end{theorem}
 
 The claim of this theorem is the original definition of homomorphism in \cite{NRS15}, and in many cases it is rather easier to work with this equivalent definition. Thus, a homomorphism $f$ of $(G, \sigma)$ to $(H, \pi)$ is built of three components: $f_1$, which decides at each vertex whether a switching is done; $f_2$, which is the vertex-mapping; $f_3$, which is the edge-mapping. Since in the case of simple graphs $f_3$ is uniquely determined by $f_2$, the mapping $f$ could be given simply as $f=(f_1,f_2)$.

 \subsection{The Extended Double Cover contruction and signed projective cubes}
 
 The notion of homomorphisms of signed graphs was first defined by B. Guenin \cite{Guenin2005} to present a bipartite analogue of a conjecture of the second author on an extension of the Four-Color Theorem. The conjecture is about mapping planar graphs to a class of graphs known as projective cubes, or folded cubes, which are closely related to hypercubes. After defining them as signed graphs, an inductive definition of projective cubes was presentend in \cite{RezaHDR}. This has lead to the introduction of the notion of the Extended Double Cover of a signed graph in \cite{NSZ20}, defined as follows. 
 
 \begin{definition}[\cite{NSZ20}]
 	\label{def-EDC}
 	Given a signed graph $(G, \sigma)$, the \emph{Extended Double Cover} of $(G, \sigma)$, denoted  ${\rm EDC}(G, \sigma)$, is defined to be the signed graph on vertex set $V^+\cup V^-$, where $V^+ := \{v^+: v \in V(G)\}$ and $V^- := \{v^-: v \in V(G)\}$. For each vertex $x$, the two vertices $x^+$ and $x^-$ are connected by a negative edge; all other edges, to be described next, are positive. If vertices $u$ and $v$ are adjacent in $(G, \sigma)$ by a positive edge, then $v^+u^+$ and $v^-u^-$ are two positive edges of ${\rm EDC}(G, \sigma)$. If vertices $u$ and $v$ are adjacent in $(G, \sigma)$ by a negative edge, then $v^+u^-$ and $v^-u^+$ are two positive edges of ${\rm EDC}(G,\sigma)$. 
 \end{definition} 
 
\begin{figure}[htpb]
	\begin{minipage}[t]{.5\textwidth}
		\centering
		\begin{tikzpicture}[>=latex,
			roundnode/.style={circle, draw=black!90, very thick, minimum size=5mm, inner sep=0pt},
			scale=1.2] 
			\node[roundnode](a1) at (0,4.5){\small $u$};
			\node[roundnode](b1) at (0,3) {\small $v$};
			\node[roundnode](c1) at (0,1.5){\small $w$};
			\node[roundnode](d1) at (0,0){\small $x$}; 
			\draw[blue, line width =1.2pt](a1)--(b1)--(c1);
			\draw[blue, line width =1pt](c1)--(d1); 
			\draw[red, line width =1pt, dashed](a1) edge[bend right] (d1);   
		\end{tikzpicture}
		\subcaption{$(C_{-4})$} 
	\end{minipage} 
	\begin{minipage}[t]{.4\textwidth}
		\centering
		\begin{tikzpicture}[>=latex,roundnode/.style={circle, draw=black!90, very thick, minimum size=5mm, inner sep=0.2pt},scale=1.2] 
		\node[roundnode](a1) at (0,4.5){\small $u^+$};
		\node[roundnode](b1) at (0,3) {\small $v^+$};
		\node[roundnode](c1) at (0,1.5){\small $w^+$};
		\node[roundnode](d1) at (0,0){\small $x^+$}; 
		\node[roundnode](a2) at (3,4.5){\small $u^-$};
		\node[roundnode](b2) at (3,3) {\small $v^-$};
		\node[roundnode](c2) at (3,1.5){\small $w^-$};
		\node[roundnode](d2) at (3,0){\small $x^-$}; 
		\draw[blue, line width =1.2pt](a1)--(b1)--(c1);
		\draw[blue, line width =1.2pt](a2)--(b2)--(c2); 
		 
		\draw[red, line width =1.2pt, dashed](a1)--(a2); 
		\draw[red, line width =1.2pt, dashed](b1)--(b2); 
				\draw[red, line width =1.2pt, dashed](a1)--(a2); 
			\draw[red, line width =1.2pt, dashed](c1)--(c2);
			\draw[blue, line width =1.2pt](c1)--(d1);
		\draw[red, line width =1.2pt, dashed](d1)--(d2); 
		   \draw[blue, line width =1.2pt](d2)--(c2);
  \draw[blue, line width =1.2pt](a1)--(d2);
  \draw[blue, line width =1.2pt](a2)--(d1);
		\end{tikzpicture}
		\subcaption{${\rm EDC}(C_{-4})$}
	\end{minipage} 
	\caption{Signed graphs $C_{-4}$ and ${\rm EDC}(C_{-4})$. Dashed (red) edges are negative.}
	\label{fig:EDCC4}
\end{figure}
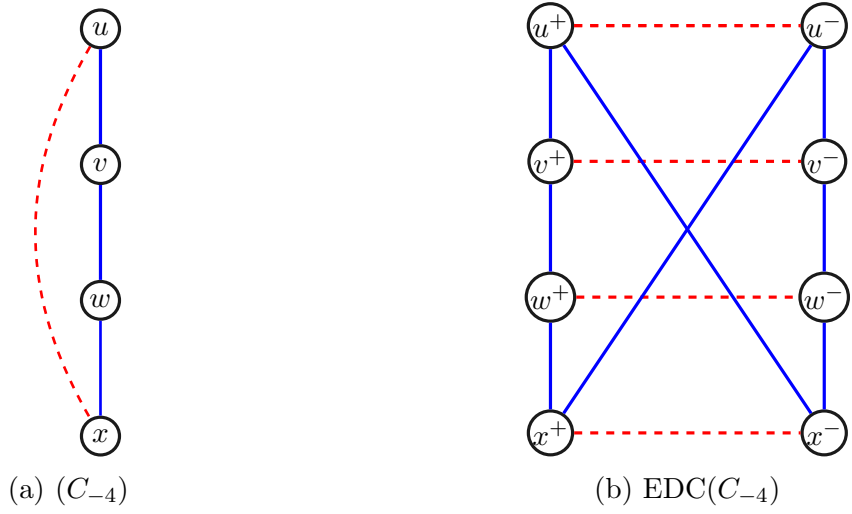
 
 A key point of this construction is that it gives a geometric shape to signs and the switching operation. Viewing vertices of $V^+$ on one side and vertices of $V^-$ on the other side, a negative edge of $(G, \sigma)$ corresponds to a twist in ${\rm EDC}(G, \sigma)$, whereas a positive edge is represented by a parallel pair of edges. For example, ${\rm EDC}(C_{-k})$, where $(C_{-k})$ is the signed graph on the cycle graph $C_k$ with an odd number of negative edges, is a M\"obius ladder, see Figure \ref{fig:EDCC4}. Switching at a vertex $x$ of $(G, \sigma)$ then corresponds to exchanging the sides of $x^+$ and $x^-$.

 Using the Extended Double Cover construction, the well-studied class of \emph{signed projective cubes} can be defined. 
 
 \begin{definition}\cite{NSZ20}\label{def:SPC}
 	The \emph{signed projective cube} of dimension $k$, denoted ${\rm SPC}(k)$, is defined inductively as follows: 
 	\begin{itemize}
 		\item ${\rm SPC}(1)$ is the digon, that is, the signed multigraph on two vertices connected by two parallel edges: one of positive sign, the other of negative sign; 
 		\item for $k\geq 2$, ${\rm SPC}(k)={\rm EDC(SPC}(k-1))$.
 	\end{itemize}
 \end{definition}
 
 We leave it to the reader to verify that ${\rm SPC}(k)$ can be directly defined as follows:
 
 \begin{proposition}\cite{NSZ20}
 	\label{pro:SPC-coordinate-def}
 	The signed graph ${\rm SPC}(k)$ is a signed graph built on vertex set $\mathbb{Z}_2^k$, where two vertices are adjacent with a positive edge if they differ in exactly one coordinate, and are adjacent with a negative edge if they differ in all coordinates. 
 \end{proposition}

 \subsection{No-homomorphism lemma and homomorphism bounds for signed graph classes}

 A given (signed) graph $(B,\pi)$ \emph{bounds} a class of (signed) graphs if every graph in the class admits a homomorphism to $(B,\pi)$, and $(B,\pi)$ is said to be \emph{a bound} for this class. We now introduce the following concept, that will be useful to express the existence of homomorphism bounds for signed graph classes.
 
 \begin{definition}
 	\label{def:C-complete}
 	Given a class $\mathcal{C}$ of signed graphs and a signed graph $(B, \pi)$, we say that $(B, \pi)$ is \emph{$\mathcal{C}$-complete} if every member $(G, \sigma)$ of $\mathcal{C}$ satisfying $g_{ij}(G, \sigma)\geq g_{ij}(B, \pi)$ for every $ij \in \mathbb{Z}_2^2$, admits a homomorphism to $(B, \pi)$.
 \end{definition}
 
 %Let $\p$ denote the class of signed planar graphs, $\SP$ denote the class of signed $K_4$-minor-free graphs.
 
 Having developed our terminology, a conjecture of the second author \cite{Naserasr2007} and B. Guenin\cite{Guenin2005}, in extension of the Four-Color Theorem, can be restated as follows. 
 
 %\begin{conjecture}\label{conj:MappingToPC(k)}
 %	Given a planar signed graph $(G, \sigma)$, if $g_{ij}(G, \sigma) \geq g_{ij} ({\rm SPC}(k))$, then $(G, \sigma)\to {\rm SPC}(k)$.
 %\end{conjecture}
 
 \begin{conjecture}\label{conj:MappingToPC(k)}
 	Let $\p$ be the class of signed planar graphs and let $k\geq 1$ be an integer. Then ${\rm SPC}(k)$ is $\p$-complete.
 \end{conjecture}
 
 In other words, the conjecture claims that when mapping signed graphs to ${\rm SPC}(k)$, the geometric condition of planarity and the necessary conditions of the no-homomorphism lemma (Lemma~\ref{lem:No-Hom-Lemma}) are, together, sufficient. Observe that, without the condition of planarity, the necessary conditions of Lemma~\ref{lem:No-Hom-Lemma} are far from being sufficient and in fact the ${\rm SPC}(k)$-homomorphism problem over the class of all signed graphs is an NP-hard problem, see \cite{BFHN17}. However, it is believed that the condition of planarity can be relaxed, e.g., a stronger conjecture is considered for signed graphs $(G, \sigma)$ where $G$ is a $K_5$-minor-free graph. An even stronger conjecture claims that the statement holds for every signed graph $(G, \sigma)$ as long as it has no $(K_5,-)$-signed minor~\cite{Guenin2005}.
 
 \subsection{Our results}
 
 The study of the relation between homomorphisms of planar (signed) graphs and the operation of Extended Double Cover is a recent development that captures some of the most prominent questions and conjectures in this area. We refer to~\cite{NSZ20} for details on this relation.
 
 In this work, in Section~\ref{sec:k4minorfree}, we study this relation when restricted to the class of $K_4$-minor-free graphs. This supports some of the above-mentioned conjectures and extends some previous works~\cite{BFN17,BFN19}. We reformulate the existing results in the unified language of signed graphs.

 We then prove that if a connected signed graph with no positive odd walk is $\SP$-complete (where $\SP$ denotes the class of $K_4$-minor-free graphs), then so is its Extended Double Cover.

 Then, in Section~\ref{sec-newbound}, towards an optimization, we find a class of nearly optimal bounds for the signed $K_4$-minor-free graphs with given girth condition. The orders of the bounds are quadratic in terms of the girth condition, which is optimal up to a constant factor and improves upon the known bounds from~\cite{BFN17,BFN19}.
 
 To better present our work, we first need to extend the terminology of signed graphs to weighted signed graphs: this will be done in Section~\ref{sec:weighted}.
 
\section{Weighted signed graphs}\label{sec:weighted}

In this section we introduce the notion of weighted signed graph, with emphasis on a special case where the weights refer to the distances and signs in a signed graph on the same set of vertices.

A \emph{weighted signed graph} $(G, \omega)$ is a graph $G$ together with an assignment $w$ of weights in $\{\pm 1, \pm 2, \ldots, \pm k\}$ (for some integer $k$) to the edges of $G$. To emphasize on the maximum absolute value, we may refer to a weighted signed graph as a $k$-weighted signed graph. Observing that 1-weighted signed graphs are simply signed graphs, we have the following extensions of this terminology. 

The \emph{length} of a walk, cycle or path in a weighted signed graph is the sum of absolute values of the weights of its edges. The sign of such structures is the product of all the signs of weights associated to the edges of said structure (considering multiplicity). Switching at a cut $(X, V\setminus X)$ is to change the sign of the weight of each edge in this cut, while the absolute value remains the same. The walk-girth of type $ij$ (for $ij\in \mathbb{Z}_2^2$) of a signed graph $(G,\sigma)$, denoted $g_{ij}(G,\sigma)$, is defined similarly as before and again remains invariant under switching. Thus, this definition similarly leads to three special classes of weighted signed graphs, that we denote by $\mathcal{C}_{ij}$, $ij\in \{01, 10, 11\}$. It will be clear from the context whether signed graphs or weighted signed graphs are considered. Thus, when speaking of weighted signed graphs, the class $\mathcal{C}_{ij}$, $ij\in \{01, 10, 11\}$, consists of those weighted signed graphs in which every closed walk is of type either $00$ or $ij$. As before, each member of $\mathcal{C}_{01}$ has a switching equivalent weighted signed graph where all weights are positive. Similarly, each member of $\mathcal{C}_{11}$ has a switching equivalent weighted signed graph where all edges are negative. Finally, each member of $\mathcal{C}_{10}$ admits a natural bipartition: having picked an arbitrary vertex $x$, all vertices connected to $x$ by a path of odd length form one part, and all vertices connected to $x$ by a path of even length (including $x$ itself) form the other part.

A homomorphism of a weighted signed graph $(G, \omega)$ to a weighted signed graph $(H, \theta)$ is defined analogously: that is, a mapping of the vertices and edges of $(G, \omega)$ to, respectively, the edges and vertices of $(H, \theta)$ which preserves the following fundamental properties, $(i)$. adjacencies, $(ii)$. incidences, $(iii)$. absolute values of weights and $(iv)$. signs of closed walks. It can be proved similarly that this is the same as a switching $(G, \omega')$ of $(G, \omega)$ after which one must preserve the sign of each edge rather than signs of closed walks. The basic no-homomorphism lemma works here as well.

\begin{lemma}\label{lem:Weighted No-Hom-Lemma}
	Given weighted signed graphs $(G, \omega)$ and $(H, \theta)$, if $(G,\omega) \to (H, \theta)$, then for each $ij\in \mathbb{Z}_2^2$, we have $g_{ij}(G,\omega)\geq g_{ij}(H, \theta)$. 
\end{lemma}

The focus of this work is on weighted signed graphs $(G, \omega)$ which have no positive closed walk of odd weight. If we assume, furthermore, that $G$ is connected, then we easily observe that all negative closed walks in $(G, \omega)$ are of the same parity. In other words either $(G,\omega) \in \mathcal{C}_{10}$ or $(G,\omega) \in \mathcal{C}_{11}$. The smallest weight of a negative closed walk in such a weighted signed graph is of importance:

\begin{definition}
	\label{g-wide}
	For a positive integer $g$, a weighted signed graph $(G,w)$ is said to be \emph{$g$-wide} if for each choice of $ij\in\mathbb{Z}_2^2$, we have $g_{ij}(G,w)\geq g_{ij}(C_{-g})$.
\end{definition}

A special family of weighted signed graphs are the ones where the weight of each edge represents the distance in a signed graph on the same set of vertices (normally the signed graph induced by edges of weight $1$ and $-1$). To make this definition formal enough we use an extended notion of distance as defined in~\cite{BFN19}, where in a signed graph, we allow two vertices to have a negative distance.

\begin{definition}
	\label{def-eq-algebraic distance}
	For a signed graph $(G,\sigma)$,  the {\em algebraic distance} between two vertices $u$ and $v$ is defined as follows:  
	\begin{equation*}
		ad_{(G,\sigma)}(u,v)=\left\{\begin{array}{ll}
			d_G(u,v), & \text{if there is a positive $(u-v)$ path of length $d_G(u,v)$}, \\
			-d_G(u,v),  & \text{otherwise.}
		\end{array}
		\right.
	\end{equation*} 
\end{definition}

Observe that the algebraic distance of $x$ and $y$ will change sign if a switching is done at one of $x$ or $y$, but not both. All other switchings preserve the algebraic distance.

\begin{definition}
 	Given a signed graph $(G, \sigma)$ and a weighted signed graph $(G',\omega)$, where $V(G')=V(G)$, we say that $(G',\omega)$ is a \emph{partial $(G, \sigma)$-distance graph} if for every edge $uv$ of $G'$, $\omega(uv)=ad_{(G,\sigma)}(u,v)$. If for every edge $xy$ of $G'$, $|\omega(x, y)| \leq  k$ we say that $(G',\omega)$ is a \emph{$k$-partial $(G, \sigma)$-distance graph}.
\end{definition}

Given a signed graph $(G, \sigma)$, the partial $(G, \sigma)$-distance graph and the Extended Double Cover of such graph, which will be introduced in Section \ref{EDC of weighted graph-I}, play important roles in this paper. 
 
\subsection{Extended Double Cover of weighted signed graphs}
\label{EDC of weighted graph-I}

We  extend the notion of Extended Double Covers of  signed graphs to Extended Double Covers of weighted signed graphs as follows:

\begin{definition}\label{def-WeightedEDC}
	Given a weighted signed graph $(G, \omega)$, the \emph{Extended Double Cover} of $(G, \omega)$, denoted  ${\rm EDC}(G,\omega)$, is defined to be the weighted signed graph on vertex set $V^+\cup V^-$, where $V^+ := \{v^+: v \in V(G)\}$ and $V^- := \{v^-: v \in V(G)\}$. Vertices $x^+$ and $x^-$ are adjacent by an edge of weight $-1$; for each pair $xy$ of adjacent vertices of $G$, there will be four more edges in ${\rm EDC}(G,\omega)$, namely $x^+y^+, x^+y^-, x^-y^+, x^-y^-$ whose weights are determined as follows. If $xy$ is an edge of weight $p>0$, then $x^+y^+$ and $x^-y^-$ are both of weight $p$ and $x^+y^-$, $x^-y^+$ are both of  weight $-(p+1)$.  If $xy$ is an edge of weight $-p<0$, then $x^+y^-$ and $x^-y^+$ are both of weight $p$ and $x^+y^+$, $x^-y^-$ are both of weight $-(p+1)$.
\end{definition}

As in the case of Extended Double Covers of signed graphs, the Extended Double Cover of a signed weighted graph adds a geometric view to the notion of switching: to switch at a vertex $v$ of $(G, \omega)$ is equivalent to switch the role of $v^+$ and $v^-$ in ${\rm EDC}(G,\omega)$.

The following is a key property of this extended notion of Extended Double Cover.

\begin{lemma}\label{lem:girthEDC}
	Given a weighted signed graph $(G, \omega)$ we have:
	\begin{itemize}
		\item $g_{01}({\rm EDC}(G,\omega))=g_{01}(G,\omega)$.
		\item $g_{10}({\rm EDC}(G,\omega))=g_{11}(G,\omega)+1$.
		\item $g_{11}({\rm EDC}(G,\omega))=g_{10}(G,\omega)+1$.
	\end{itemize}
\end{lemma}

\begin{proof}
	All three claims are consequences of the following observation. Given a signed closed walk $W$ of $(G, \omega)$ there is a natural association with two closed walks, denoted ${\rm EDC}^+(W)$   and ${\rm EDC}^-(W)$ in ${\rm EDC}(G,w)$. If the starting vertex of $W$ is $x$, then the starting point of ${\rm EDC}^+(W)$ is $x^+$ and that of ${\rm EDC}^-(W)$ is $x^-$. 
	The descriptions of ${\rm EDC}^+(W)$, and that of ${\rm EDC}^-(W)$ except for the starting point, are the same. 
	If the $i^{th}$ vertex of $W$ is $v$, then the $i^{th}$ vertex  of ${\rm EDC}^+(W)$ is one of $v^+$ or $v^-$, the choice of which is implied from the following procedure. 
	
	Assume that at step $i$ of $W$ we are at vertex $v$ and that $v'\in \{v^+, v^-\}$  is determined as the $i^{th}$ vertex of ${\rm EDC}^+(W)$. Let $u$ be the next vertex on $W$. Then choose the vertex $u' \in \{u^+, u^-\}$ as follows: if the edge $vu$ in $W$ is positive, then $u'$ has the same sign as $v'$, otherwise it has the opposite sign. If $W$ is a positive closed walk, this process ends with $x^+$ and we have ${\rm EDC}^+(W)$. But  if $W$ is a negative closed walk, this process ends with $x^-$ in which case we must add the negative edge $x^+x^-$ in order to have a closed walk ${\rm EDC}^-(W)$. In this case, ${\rm EDC}^-(W)$ is a negative closed walk of length 1 more than that of $W$, thus of different parity. 
	
	It is easily observed that each closed walk of ${\rm {\rm EDC}}(G,\omega)$ that uses at most one negative edge is either of the form ${\rm EDC}^+(W)$ or of the form ${\rm EDC}^-(W)$ for a closed walk $W$ of $(G, \omega)$. Furthermore, if two edges of the form $v^+v^-$ are used, then we can create a closed walk of shorter length which is of the same sign and the same parity.  Thus the minimum length closed walks of a given type can use at most one edge of type $v^+v^-$. The three claims then follow.	
\end{proof}
 
Following the proof of Lemma \ref{lem:girthEDC}, we have the following lemma (note that each vertex can be viewed as the starting vertex).

\begin{lemma}
	\label{lemma:closed p+i-walk in EDC}
	Let $(G, \sigma)$ be a signed graph with $x,y$ two vertices of $(G,\sigma)$ in a cycle $C$ of length $g$ in $(G,\sigma)$. If $C$ is positive, then each of the pairs $(x^+,y^+)$ and $(x^-,y^-)$ is also in a positive cycle of length $g$ in ${\rm EDC}(G, \sigma)$. If $C$ is negative, then each of the pairs $(x^+,y^+)$, $(x^-,y^-)$, $(x^+,y^-)$ and $(x^-,y^+)$ is in a negative cycle of length $g+1$ in ${\rm EDC}(G, \sigma)$. 
\end{lemma}
 
\subsection{Extended Double Cover of signed graphs}
\label{sec:EDC of signed graphs}

In this section, we focus on the Extended Double Covers of signed graphs and we prove some useful properties. From Definition~\ref{def-EDC}, we have the following observation.

%The following proposition is useful in our work.
%\begin{proposition}[{\cite[Proposition 26]{NSZ20}}]\label{prop:EDC-switching-invariant}
%	For a graph $G$ with two switching equivalent signatures $\sigma$ and $\sigma'$, ${\rm EDC}(G, \sigma)$ and ${\rm EDC}(G, \sigma')$ (viewed as edge-colored graphs) are isomorphic.
%\end{proposition}

\begin{observation}\label{obs:walk in EDC}
	Let $(G, \sigma)$ be a signed graph with $x,y$ two vertices of $(G,\sigma)$ and an ($x-y$)-walk $W$  of length $p$ in $(G,\sigma)$. If $W$ is positive, then, in ${\rm EDC}(G, \sigma)$, there exisst an ($x^+-y^+$)-walk and an $(x^--y^-$)-walk, both of which are positive and of length $p$, and an ($x^+-y^-$)-walk and an ($x^--y^+$)-walk, both of which are negative and of length $p+1$. 
	If $W$ is negative, then, in ${\rm EDC}(G, \sigma)$, there exist an ($x^+-y^-$)-walk and an ($x^--y^+$)-walk, both of which are positive and of length $p$, and an ($x^+-y^+$)-walk and an ($x^--y^-$)-walk, both of which are negative and of length $p+1$. 
\end{observation}

The following computes the algebraic distance between two vertices in ${\rm EDC}(G,\sigma)$. 

\begin{proposition}
	\label{pro:ad in EDC}
	Let $(G,\sigma)$ be a $g$-wide signed graph with $x,y$ two of its vertices that are in a common negative cycle of length $g$, and $ad_{(G,\sigma)}(x,y)=p$. Then, the following statements hold.
	\begin{itemize}
		\item If $p > 0$,  then $ad_{{\rm EDC}(G,\sigma)}(x^+,y^+)=ad_{{\rm EDC}(G,\sigma)}(x^-,y^-)=p$. Moreover:
		\begin{itemize}
			\item if $d_{G}(x,y)=\lfloor\frac{g}{2}\rfloor$, then $ad_{{\rm EDC}(G,\sigma)}(x^+,y^-)=ad_{{\rm EDC}(G,\sigma)}(x^-,y^+)=g-p = \lceil\frac{g}{2}\rceil$;
			\item otherwise, $ad_{{\rm EDC}(G,\sigma)}(x^+,y^-)=ad_{{\rm EDC}(G,\sigma)}(x^-,y^+)=-p-1$. 
		\end{itemize}
		\item 	If $p < 0$,  then $ad_{{\rm EDC}(G,\sigma)}(x^+,y^-)=ad_{{\rm EDC}(G,\sigma)}(x^-,y^+)=-p$. Moreover:
		\begin{itemize}
			\item if $d_{G}(x,y)=\lfloor\frac{g}{2}\rfloor$, then $ad_{{\rm EDC}(G,\sigma)}(x^+,y^+)=ad_{{\rm EDC}(G,\sigma)}(x^-,y^-)=g+p=\lceil\frac{g}{2}\rceil$;
			\item otherwise, $ad_{{\rm EDC}(G,\sigma)}(x^+,y^+)=ad_{{\rm EDC}(G,\sigma)}(x^-,y^-)=p-1$.
		\end{itemize} 
	\end{itemize}  
\end{proposition}

\begin{proof} Without loss of generality, by the symmetries of ${\rm EDC}(G,\sigma)$, we only focus on the pairs $(x^+,y^+)$ and $(x^+,y^-)$. As $ad_{(G,\sigma)}(x,y)=p$, and $x,y$ are in a negative cycle of length $g$, by the definition of algebraic distance, we know that $|p| \leq \lfloor\frac{g}{2}\rfloor$. Since $(G,\sigma)$ is $g$-wide, by Lemma \ref{lem:girthEDC}, ${\rm EDC}(G,\sigma)$ is $(g+1)$-wide.
	
	First assume that $p >0$. Then, there exist a positive $(x-y)$-path of length $d_G(x,y)=p$ and a negative $(x-y)$-path of length $g-p$ in $(G,\sigma)$. By Observation~\ref{obs:walk in EDC}, in  ${\rm EDC}(G,\sigma)$, there exist a positive  $(x^+-y^+)$-path $P_1$ of length $p$, a negative $(x^+-y^-)$-path $P_2$ of length $p+1$ and  a positive $(x^+-y^-)$-path $P_3$ of length $g-p$.  So the distances between $x^+$ and $y^+$, $x^+$ and $y^-$ in the underlying graph of ${\rm EDC}(G,\sigma)$ are $p$ and $\min\{p+1,g-p\}$, respectively. By Definition~\ref{def-eq-algebraic distance}, $ad_{{\rm EDC}(G,\sigma)}(x^+,y^+)=p$. If $p=\lfloor\frac{g}{2}\rfloor$, then $p+1=\lfloor\frac{g}{2}\rfloor+1\geq \lceil\frac{g}{2}\rceil= g-p$, hence $\min\{p+1,g-p\}=g-p$  and $ad_{{\rm EDC}(G,\sigma)}(x^+,y^-)=g-p=\lceil\frac{g}{2}\rceil$. Otherwise, we know that $p \leq \lfloor\frac{g}{2}\rfloor-1$, which implies that the distance between $x^+$ and $y^-$ in the underlying graph of ${\rm EDC}(G,\sigma)$ is $\min\{p+1,g-p\}=p+1$. Note that there is no positive $(x^+-y^-)$ path of length $p+1$, since otherwise such path, together with $P_2$, consists of a cycle of length $2(p+1)=2\lfloor\frac{g}{2}\rfloor \leq g$, a contradiction to the fact that  ${\rm EDC}(G,\sigma)$ is $(g+1)$-wide. Thus, by Definition~\ref{def-eq-algebraic distance}, $ad_{{\rm EDC}(G,\sigma)}(x^+,y^-)=-(p+1)$.
	
	Now assume that $p < 0$. By similar arguments as above, there exist a negative $(x-y)$-path of length $d_G(x,y)=-p$ and a positive $(x-y)$-path of length $g+p$ in $(G,\sigma)$. By Observation~\ref{obs:walk in EDC}, there exist a positive  $(x^+-y^-)$-path $P_4$ of length $-p$, a negative $(x^+-y^+)$-path $P_5$ of length $-p+1$ and  a positive $(x^+-y^+)$-path $P_6$ of length $g+p$. Recall that ${\rm EDC}(G,\sigma)$ is $(g+1)$-wide and $p \geq -\lfloor\frac{g}{2}\rfloor$, so the distances between $x^+$ and $y^+$, $x^+$ and $y^-$ in the underlying graph of ${\rm EDC}(G,\sigma)$ are $\min\{-p+1,g+p\}$ and $-p$, respectively. Thus, $ad_{{\rm EDC}(G,\sigma)}(x^+,y^-)=-p$. If $p=-d_{G}(x,y)=-\lfloor\frac{g}{2}\rfloor$, then $-p+1= \lfloor\frac{g}{2}\rfloor+1 \geq \lceil \frac{g}{2}\rceil=g+p$, so $ad_{{\rm EDC}(G,\sigma)}(x^+,y^+)=g+p=\lceil\frac{g}{2}\rceil$. Otherwise, $p \geq -\lfloor\frac{g}{2}\rfloor +1$, hence $\min\{-p+1,g+p\}=-p+1$. Also observe that there is no positive $(x^+-y^+)$-path of length $-p+1$, since otherwise such path together with $P_5$ consist of a negative cycle of length $2(-p+1) \leq 2\lfloor\frac{g}{2}\rfloor \leq g$, a contradiction to the fact that  ${\rm EDC}(G,\sigma)$ is $(g+1)$-wide. Thus $ad_{{\rm EDC}(G,\sigma)}(x^+,y^+)=-(-p+1)=p-1$.
\end{proof}
 
% \subsection{Girth-transformed $(G,\sigma)$-distance graphs}

\section{$K_4$-minor-free graphs}\label{sec:k4minorfree}

A \emph{$2$-tree} is a graph that can be built from the complete graph $K_2$ in a sequence $G_0=K_2, G_1,\ldots, G_t$ where $G_i$ is obtained from $G_{i-1}$ by adding a new vertex joined to two adjacent vertices of $G_{i-1}$, thus forming a new triangle. A partial $2$-tree is a subgraph of a $2$-tree. 
It is well-known that a graph is $K_4$-minor-free if and only if it is a \emph{partial $2$-tree} (see for example~\cite{DR17}). The class of $K_4$-minor free graphs is also known as the class of \emph{series-parallel graphs}, see for example \cite{BODLAENDER19981}. Thus we will use the abbreviation $\SP$ to denote this class of graphs.

As observed from the definition of 2-trees, the triangle is the building block of edge-maximal $K_4$-minor free graphs. When a girth condition is imposed on a signed $K_4$-minor free graph  $(G,\sigma)$, then $G$ will no longer be edge-maximal, but rather a partial 2-tree. To take advantage of the structure of such signed graphs then, in \cite{BFN17} and \cite{BFN19}, weighted 2-trees are employed. Next we present these techniques in a uniform language of signed graphs (with no positive walk of odd length). We rather use the terminology developed in \cite{CN20} while extending it to signed graphs.

\subsection{Weighted triangles and $g$-wideness}

Given a positive integer $g$ with $g\geq 3$, and three integers $p$, $q$ and $r$ satisfying $1\leq |p|, |q|, |r|\leq g-1$, the signed graph $T_g(p,q,r)$ is built as follows. Let $C_{g,p}$ be a negative cycle of length $g$ with a selected pair $x_1$ and $y_1$ of vertices such that one of the two $x_1-y_1$ paths in $C_{g,p}$ is of length $|p|$ and has the same sign as $p$, and the other, which is of length $g-|p|$, has the opposite sign as $p$. Define $C_{g,q}$ similarly where selected vertices $y_2$ and $z_1$ are connected by a $|q|$-path and $C_{g,r}$ with selected vertices $z_2$ and $x_2$ which are connected by an $|r|$-path. We define the signed graph  $T_g(p,q,r)$ to be the signed graph obtained from $C_{g,p}$, $C_{g,q}$ and $C_{g,r}$ by identifying $x_1$ and $x_2$ (to form the new vertex $x$), $y_1$ and $y_2$ (to form the new vertex $y$), $z_1$ and $z_2$ (to form the new vertex $z$).  
See the left picture in Figure~\ref{Examples of Tgpqr} for an example.

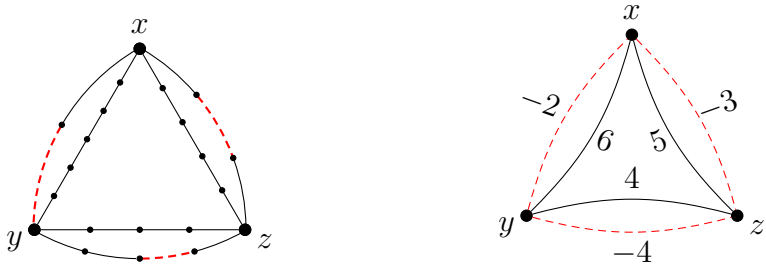
\begin{figure}[H]
		\centering 
			\begin{minipage}{0.5\textwidth} 
					\centering 
			\begin{tikzpicture}[>=latex,	
				Bignode/.style={circle, draw=black,fill=black, minimum size=1.5mm, inner sep=0pt},
				roundnode/.style={circle, draw=black,fill=black, minimum size=0.7mm, inner sep=0pt}]	
				\node [Bignode] (z) at (330:1.6){};
				\node [Bignode] (x) at (90:1.6){};	
				\node [Bignode] (y) at (210:1.6){}; 
				\node at (330:1.9){$z$};
				\node at (90:1.9){$x$};
				\node at (210:1.9){$y$}; 
				\draw (x)--(z)
				\foreach \t in {1/5,2/5,3/5,4/5}{pic [draw,pos=\t] {code={\node [roundnode]{};}}};
				\draw ($(y)+(59.5:2.7833)$) arc (59.5:40.5:2.7833);
				\draw[red,thick,densely dashed] ($(y)+(39.5:2.7833)$) arc (39.5:20.5:2.7833);
				\draw ($(y)+(19.5:2.783)$) arc (19.5:0.5:2.783);
				\node  [roundnode] at ($(y)+(40:2.783)$) {}; 
				\node  [roundnode] at ($(y)+(20:2.783)$) {}; 	
				
				\draw (x)--(y)
				\foreach \t in {1/6,2/6,3/6,4/6,5/6}{pic [draw,pos=\t] {code={\node [roundnode]{};}}};;
				\draw ($(z)+(120.5:2.783)$) arc (120.5:149.5:2.783);
				\draw[red,thick,densely dashed] ($(z)+(150.5:2.783)$) arc (150.5:179.5:2.783);
				\node  [roundnode] at ($(z)+(150:2.783)$){}; 
				
				\draw (y)--(z)
				\foreach \t in {1/4,2/4,3/4}{pic [draw,pos=\t] {code={\node [roundnode]{};}}};
				
				\draw ($(x)+(241:2.783)$) arc (241:269:2.783);
				
				\draw[red,thick,densely dashed] ($(x)+(270.5:2.783)$) arc (270.5:284.5:2.783);
				\draw ($(x)+(285.5:2.783)$) arc (285.5:299.5:2.783);
				\node  [roundnode] at ($(x)+(285:2.783)$){}; 
				\node [roundnode] at ($(x)+(255:2.783)$){}; 
				\node  [roundnode] at ($(x)+(270:2.783)$){}; 
				
				\node [Bignode] (z) at (330:1.6){};
				\node [Bignode] (x) at (90:1.6){};	
				\node [Bignode] (y) at (210:1.6){}; 
			\end{tikzpicture}  
		\end{minipage} 
	\begin{minipage}{0.4\textwidth} 
			\begin{tikzpicture}[>=latex,	
			Bignode/.style={circle, draw=black,fill=black, minimum size=1.5mm, inner sep=0pt}]	
			\node [Bignode] (z) at (330:1.6){};
			\node [Bignode] (x) at (90:1.6){};	
			\node [Bignode] (y) at (210:1.6){}; 
		 
			\draw (y) [bend left =15] to (z); 
			\draw[red,densely dashed] (y) [bend right =15] to (z); 
			
			\draw[red,densely dashed] (x) [bend left =15] to (z); 
			\draw  (x) [bend right = 15] to (z);
			
			\draw[red,densely dashed] (y) [bend left = 15] to (x); 
			\draw (y) [bend right = 15] to (x);
			
			\node[rotate=20] at (30:1.3) {$-3$};
			\node[rotate=-20] at (150:1.4){$-2$};
			\node at (270:1.3){$-4$};
			
			\node[rotate=20] at (30:0.4){$5$};
			\node[rotate=-20] at (150:0.4){$6$};
			\node at (270:0.3){$4$};
			
			\node at (330:1.9){$z$};
			\node at (90:1.9){$x$};
			\node at (210:1.9){$y$}; 
		\end{tikzpicture} 
	\end{minipage}
	\caption{$T_g(p,q,r)$ with $g=8$, $p=6$, $q=-4$, and $r=-3$. Dashed (red) edges are negative.}	
	\label{Examples of Tgpqr}
\end{figure}
 
We have a few immediate observations: $(i)$. If $p'$ is the integer satisfying $|p'|=g-|p|$ and $pp'<0$,  then $C_{g,p}$ and $C_{g, p'}$ are identical and thus $T_g(p',q,r)$ is isomorphic to $T_g(p,q,r)$. $(ii)$. A switching at an internal vertex of an ($x-y$)-path would result in a different presentation of $T_g(p,q,r)$, but up to a switching they are isomorphic. $(iii)$. A switching at $x$, $y$ or $z$ results, respectively, in $T_g(-p,q,-r)$, $T_g(-p,-q,r)$ and $T_g(p,-q, -r)$ which are also isomorphic $T_g(p,q,r)$ up to a switching. 

Note that $T_g(p,q,r)$ could be simply presented as a weighted triangle (multiple edges are allowed here), see the right picture in Figure~\ref{Examples of Tgpqr}. But considering $(i)$, between $p$ and $p'$, we would normally choose the one whose value is positive. Therefore, we let $\bigtriangleup(p,q,r)$ denote a weighted triangle whose edges are of weight $p$, $q$ and $r$, and extend the definition of $g$-wideness as follows.

\begin{definition}
	Given  a positive integer $g$, $g\geq 3$, and three integers $p$, $q$ and $r$ satisfying $1\leq |p|, |q|, |r|\leq g-1$, we say a weighted triangle $\bigtriangleup(p,q,r)$, or the triple $(p,q,r)$, is $g$-wide if $T_g(p,q,r)$ is $g$-wide.
\end{definition}

In other words, $\bigtriangleup(p,q,r)$ is $g$-wide if $T_g(p,q,r)$ satisfies the following two conditions:

\begin{itemize}
	\item There are no positive odd cycles in $T_g(p,q,r)$.
	\item Each of the negative cycles of $T_g(p,q,r)$ is of the same parity as $g$ and is, furthermore, of length at least $g$. 
\end{itemize}

In our work we will need to consider properties of triangles and edges. For a uniform writing, in the definition of a $g$-wide triple $(p,q,r)$ we may allow them to be 0 as well. If $p=0$, then in the construction of $T_{g}(p,q,r)$ the vertices $x$ and $y$ are identified and a negative cycle of length $g$ is added on the identified vertex which we may ignore. Then for $\bigtriangleup(0,q,r)$ to be $g$-wide the positive paths corresponding to $q$ and $r$ must be of the same length. Thus we may assume $q=r$. Therefore, in the rest of this work $T(0,r,r)$, $r\leq g-1$, is $g$-wide. Triples of the form $T(0,r,r)$ will, in essence of it, represent the essential edges of the weighted graphs we will work with where $r$ would be the weight of the corresponding edge.

That $\bigtriangleup(p,q,r)$ is $g$-wide depends only on the values of $p$, $q$, $r$ and $g$. We have already seen that when $p=0$, triples of the form $(0,r,r)$ are the only $g$-wide triples. For a  triple satisfying $1\leq |p|, |q|, |r|\leq g-1$ there are a number of ways to check if it is $g$-wide. In this work we will use the test provided in the next proposition.

\begin{proposition}
	\label{g-good triples}
	Given integers $g$, $p$, $q$ and $r$ satisfying $1\leq |p|, |q|, |r|\leq g-1$, the following statements hold. 
	\begin{enumerate}[label= {(\arabic*)},ref=(\theenumi)]
	\item \label{g-good-product-positive} If $pqr>0$, then the weighted triangle $\bigtriangleup(p,q,r)$ is g-wide if and only if $|p|+|q|+|r| \equiv 0 \pmod 2$ and  $\max\{2|p|,2|q|,2|r|\} \leq |p|+|q|+|r| \leq 2g$. 
		
		\item \label{g-good-product-negative} If $pqr<0$, then  the weighted triangle $\bigtriangleup(p,q,r)$ is g-wide if and only if $|p|+|q|+|r| \equiv g \pmod  2$ and  $g \leq |p|+|q|+|r| \leq g+\min\{2|p|,2|q|,2|r|\}$.
	\end{enumerate} 
\end{proposition}

\begin{proof}~\\
	\textbf{\boldmath{$(1)$ $pqr > 0$}.} There are exactly four positive cycles in $T_g(p,q,r)$, and their lengths, which are of the same parity, are: $|p|+|q|+|r|$, $g-|p|+g-|q|+|r|$, $g-|p|+|q|+g-|r|$, and $|p|+g-|q|+g-|r|$.

	Suppose first that $\bigtriangleup(p,q,r)$ is $g$-wide. By the definition, we have $g_{ij}(T_{g}(p,q,r)) \geq g_{ij}(C_{-g})$ and since $g_{01}(C_{-g}) = \infty$, there is no positive odd cycle, in other words $|p|+|q|+|r|$ is even. Except the four cycles we mentioned above, all the other cycles in $T_{g}(p,q,r)$ are negative. Moreover, the three  cycles  containing exactly two of $\{x,y,z\}$ are all of length $g$.  The four negative cycles  containing all three vertices $x,y,z$ are of length $g-|p|+g-|q|+g-|r|$, $g-|p|+|q|+|r|$, $|p|+g-|q|+r$, $|p|+|q|+g-|r|$. Since the negative girth of $T_g(p,q,r)$ is $g$, we have: $g-|p|+g-|q|+g-|r|\geq g$ which is to say $|p|+|q|+|r|\leq 2g$. Assuming, without loss of generality, that $\max\{|p|,|q|,|r|\}=|p|$, the condition $g-|p|+|q|+|r|\geq g$ implies that $\max\{2|p|,2|q|,2|r|\}=2|p|\leq |p|+|q|+|r|$.
	
	Conversely, assume that $|p|+|q|+|r|$ is even and $\max\{2|p|,2|q|,2|r|\} \leq |p|+|q|+|r|\leq 2g$. We shall show that $g_{ij}(T_{g}(p,q,r)) \geq g_{ij}(C_{-g})$, for any $ij \in \{01,10,11\}$. The case $ij=01$ follows from the fact that $|p|+|q|+|r|$ is even and the argument discussed in the first paragraph. The three cycles containing exactly two of $x,y,z$ are of length $g$, and all of them are negative.  The four negative cycles containing $x,y,z$ are of length $g-|p|+g-|q|+g-|r|$, $g-|p|+|q|+|r|$, $g-|q|+|p|+|r|$, $g-|r|+|p|+|q|$. By the assumptions, all of these four values are of the same parity as $g$, and also are at least $g$, which implies that $g_{1j}(T_{g}(p,q,r)) \geq g_{1j}(C_{-g})$, for each $j \in \{0,1\}$.
	
	\medskip\noindent\textbf{\boldmath{$(2)$ $pqr < 0$}.}  Let $p'=-\frac{p}{|p|}(g-|p|)$, $q'=-\frac{q}{|q|}(g-|q|)$,  $r'=-\frac{r}{|r|}(g-|r|)$, so $pqr<0$ if and only if $p'q'r'>0$.  Since $T_g(p,q,r)$ is isomorphic to  $T_g(p',q',r')$, $\bigtriangleup(p,q,r)$ is $g$-wide if and only if $T(p',q',r')$ is $g$-wide. Therefore, by \ref{g-good-product-positive}, $\bigtriangleup(p,q,r)$ is $g$-wide if and only if $|p'|+|q'|+|r'| = 0 \pmod 2$ and $\max\{2|p'|,2|q'|,2|r'|\} \leq |p'|+|q'|+|r'| \leq 2g$. Equivalently, $g-|p|+g-|q|+g-|r| = 0 \pmod 2$ and $\max\{2g-2|p|,2g-2|q|,2g-2|r|\} \leq g-|p|+g-|q|+g-|r| \leq 2g$, after simplification, we have $|p|+|q|+|r| = g \pmod  2$ and  $g \leq |p|+|q|+|r| \leq g+\min\{2|p|,2|q|,2|r|\}$.
\end{proof}

\subsection{A test for $\mathcal{SP}$-completeness}

 We denote by $\mathcal{L} _g$ the set of ordered triples $(p,q,r)$, satisfying $|p|, |q|, |r|\leq g-1$ and such that $\bigtriangleup(p,q,r)$ is $g$-wide. Observe that because of the condition $|p|, |q|, |r|\leq g-1$ we have less than $8g^3$ non-isomorphic weighted triangles (or edges) $\bigtriangleup(p,q,r)$, thus $[\mathcal{L}_g|\leq 8g^3$.

Recall that for each $p$ with $1\leq |p| \leq g-1$, if $p'$ is the integer satisfying $|p'|=g-|p|$ and $pp'<0$, then $T_g(p', q,r)$ is the same as $T_g(p,q,r)$. Thus  a triple $(p,q,r)$ can be represented in $\mathcal{L} _g$ in 8 possible ways among which there is a unique presentation where $p,q,r\geq 1$. 

Similarly, recall that in the definition of $(G,\sigma)$-distance graph, for each weighted edge, the weight represents the algebraic distance between the two endpoints in $(G,\sigma)$, which could be either positive or negative. Another special weighted signed graph is obtained from $(G,\sigma)$ by using only positive weights, which represent the length of positive paths joining pairs of vertices on a shortest negative cycle of $(G,\sigma)$ containing both of them. For this point, we define the following.

\begin{definition}
	\label{def-girth transformed}
	Given a $g$-wide signed graph $(G, \sigma)$ and a weighted signed graph $(G',\omega)$, where $V(G')=V(G)$ and $G'$ is such that for each edge $xy$ of $G'$, the pair $x,y$ is in a negative cycle of length $g$ in $(G, \sigma)$, we say that $(G',\omega)$ is a \emph{girth-transformed $(G,\sigma)$-distance graph} if for every edge $uv$ of $G'$, $\omega(uv)=f_g(ad_{(G,\sigma)}(u,v))$, where $f_g$ is defined on $-\gc+1 \leq x \leq \gl$, $x \neq 0$, as following: 
	$$f_g(x) = 
	\begin{cases}
		x, & \text{if $x >0$,} \\
		g+x, &  \text{otherwise,}
	\end{cases}$$	 
 	If for every edge $xy$ of $G'$, $\omega(uv) \leq  k$, we say that $(G',\omega)$ is a \emph{$k$-partial girth-transformed $(G, \sigma)$-distance graph}.
\end{definition}
  
Now, with this transformation, we can address weighted signed graphs with only positive weights. Also, some known theorems can be also restated in this language. The following definition is a restatement of the ``all $g$-good property" in \cite{BFN17} and \cite{BFN19}.
\begin{definition}
	\label{def:triangle g closed}
	Given a $g$-wide weighted signed graph $(G, \omega)$ satisfying $1 \leq \omega(e) \leq g-1$ for every edge $e$, a set $\mathcal T$ of triangles of $G$ is said to be \emph{$g$-closed} if the following condition is satisfied:   
	
	Denoting by $\mathcal{E}$ the set of weighted edges of the triangles in $\mathcal T$, for each edge $xy \in \mathcal{E}$ (assuming $\omega(xy)=p$) and for each triple $(p, q, r)\in \mathcal{L} _g$, there is a triangle $xyz \in  \mathcal T$ such that   $\omega(zx)=q$ and $\omega(zy)=r$, or $\omega(zx)=g-q$ and $\omega(zy)=g-r$. 
\end{definition}

\begin{remark}
	The only difference between all $g$-good property and $g$-closed is that in the former one, (i). $1 \leq |w(e)| \leq  \lfloor \frac{g}{2} \rfloor$; (ii). the last condition is   $w(zx) = q$, $w(zy) = r$ or $w(zx) = -q$, $w(zy) = -r$.  
\end{remark}

Note that in the condition above, if $q \neq r$, then the order of $p,q,r$ matters. To be precise, in such a case, say for $(p,r,q)$, while the definition implying existence of a vertex $z$ satisfying $\omega(zx)=q$ and $\omega(zy)=r$, or $\omega(zx)=g-q$ and $\omega(zy)=g-r$ by consider the triple $(p,r,q)$  instead of $(p,q,r)$ there must also be a vertex $z'$ satisfying  $\omega(xz')=r$ and $\omega(yz')=q$, or $\omega(xz')=g-r$ and $\omega(yz')=g-q$.

The following then is a uniform restatement of results of \cite{BFN17} and \cite{BFN19}.  (see also \cite{CN20})
 
\begin{theorem} 
	\label{thm: old bounds iff}
	A $g$-wide signed graph $(B,\pi)$ is $\SP$-complete if and only if there exists a
	$\gl$-partial $(B,\pi)$-distance graph with a nonempty set $\mathcal{T}$ having  all $g$-good property.  
\end{theorem} 

Then we have the following observation.

\begin{observation}
	\label{obs:all g good property- x and y in a nagative g-cyce}
	Let $(B,\pi)$ be a $g$-wide signed graph, $(B',\omega)$ be a $\gl$-partial $(B,\pi)$-distance graph with a nonempty set $\mathcal{T}$ having all $g$-good property. Then for every edge $xy$ in $T \in \mathcal{T}$, $x$ and $y$ are in a negative cycle of length $g$ in $(B,\pi)$.
\end{observation}
\begin{proof}
	Let $\omega(x,y)=p$ with $|p| \leq \lfloor \frac{g}{2}\rfloor$. 
	By Proposition~\ref{g-good triples}(2), if $p>0$, then $(p,-\lfloor\frac{g-p}{2}\rfloor, \lceil \frac{g-p}{2}\rceil)  \in \mathcal{L}_g$, and if $p<0$, then $(p,\lfloor \frac{g+p}{2}\rfloor, \lceil \frac{g+p}{2}\rceil) \in \mathcal{L}_g$ . Thus in each case, there exists a vertex $z$ in $B$ such that  
	$$|\omega(xy)|+|\omega(zx)|+|\omega(zy)|=|p|+\left\lfloor \frac{g-|p|}{2}\right\rfloor+\left\lceil \frac{g-|p|}{2}\right\rceil = g,$$ and  $\omega(xy)\cdot \omega(zx)\cdot \omega(zy) <0$, which implies that there is a negative cycle of length $g$ in $(B,\pi)$ containing both $x$ and $y$.  
\end{proof}
 
 Then Theorem~\ref{thm: old bounds iff} and Observation \ref{obs:all g good property- x and y in a nagative g-cyce} can be restated with only positive weights as follows, respectively.

\begin{theorem}
	\label{thm:bounds if and only if}
	A $g$-wide signed graph $(B,\pi)$ is $\SP$-complete if and only if there exists a $(g-1)$-partial  girth-transformed  $(B,\pi)$-distance graph $(B', \omega)$ which has a nonempty $g$-closed set $\mathcal{T}$ of triangles. 
\end{theorem}
 
\begin{observation}
	\label{obs:g-closed - x and y in a nagative g-cyce}
	Let $(B,\pi)$ be a $g$-wide signed graph, $(B',\omega)$ be a  $(g-1)$-partial  girth-transformed  $(B,\pi)$-distance graph which has a nonempty $g$-closed set $\mathcal{T}$ of triangles. Then for every edge $xy$ in $\mathcal{T}$, $x$ and $y$ are in a negative cycle of length $g$ in $(B,\pi)$.
\end{observation}

Now, we are ready to state and prove our main theorem.
 
\begin{theorem}
	For any positive integer $g$, if a $g$-wide signed graph $(B, \pi)$ is $\mathcal{SP}$-complete, then so is ${\rm EDC}(B, \pi)$.
\end{theorem}

 \begin{proof}
 	Let $(B, \pi)$ be a $g$-wide signed graph which is $\mathcal{SP}$-complete.
 	By Theorem~\ref{thm:bounds if and only if}, there exists a $g$-partial $(B,\pi)$-distance graph $(B', w)$ which has a nonempty and $g$-closed set $\mathcal{T}$ of triangles whose vertex set is $V$ and (weighted) edges form the set $\mathcal{E}$. We first define the weighted signed graph  $(\widehat{B},\widehat{w})$ on vertex set $V^+\cup V^-$, where $V^+ := \{v^+: v \in V\}$ and $V^- := \{v^-: v \in V\}$ as follows: for each $v \in V$, vertices $v^+$ and $v^-$ are joined by an edge with $\widehat{w}(v^+v^-)=g$. If $uv \in \mathcal{E}$ with weight $w(uv)$, then we add four edges $\{u^+v^+,u^+v^-,u^-v^-,u^-v^+\}$ with $\widehat{w}(u^+v^+)=\widehat{w}(u^-v^-)=w(uv)$, and $\widehat{w}(u^+v^-)=\widehat{w}(u^-v^+)=g-w(uv)$.
 	
 	\begin{claim}
 		  $(\widehat{B},\widehat{w})$ is a $g$-partial girth-transformed ${\rm EDC}(B,\pi)$-distance graph.
 	\end{claim}
 	\begin{proof}
 		Let $x,y$ be two vertices of $B'$ forming an edge in $\mathcal{E}$, by Observation \ref{obs:g-closed - x and y in a nagative g-cyce}, $x,y$ are contained in a negative cycle of length $g$ in $(B, \pi)$.
 		By Lemma~\ref{lemma:closed p+i-walk in EDC}, $x^{\alpha}$ and $y^{\beta}$ are in a negative cycle of length $g+1$ in  ${\rm EDC}(B,\pi)$ for any $\alpha,\beta \in \{+,-\}$. Thus, it suffices to show that $\widehat{\omega}(x^\alpha y^\beta)=f_{g+1}(ad_{{\rm EDC}(B,\pi)}(x^\alpha, y^\beta))$.

 		For each vertex $x$ in $B$,  $ad_{{\rm EDC}(B,\pi)}(x^+,x^-)=-1<0$, so   
 		$f_{g+1}(ad_{{\rm EDC}(B,\pi)}(x^+,x^-))=g+1-1=g=\widehat{\omega}(x^+x^-)$.
 		
 		Assume $x$ and $y$ are two vertices in $B$
 		with $ad_{(B,\pi)}(x,y)=p>0$. Thus $\omega(xy)=p$. By Proposition~\ref{pro:ad in EDC},  $ad_{{\rm EDC}(B,\pi)}(x^+,y^+)=ad_{{\rm EDC}(B,\pi)}(x^-,y^-)=p$.  Then we have $f_{g+1}(ad_{{\rm EDC}(B,\pi)}(x^+,y^+))=f_{g+1}(ad_{{\rm EDC}(B,\pi)}(x^-,y^-))=p=\omega(xy)=\widehat{\omega}(x^+y^+)=\widehat{\omega}(x^-y^-)$. 
 		If $d_{B}(x,y)=\lfloor\frac{g}{2}\rfloor$, then $\omega(xy)=p=\lfloor\frac{g}{2}\rfloor$.
 		By Proposition~\ref{pro:ad in EDC}, $ad_{{\rm EDC}(B,\pi)}(x^+,y^-)=ad_{{\rm EDC}(B,\pi)}(x^-,y^+)=\lceil\frac{g}{2}\rceil$, so we have
 		$$f_{g+1}(ad_{{\rm EDC}(B,\pi)}(x^+,y^-))=f_{g+1}(ad_{{\rm EDC}(B,\pi)}(x^-,y^+))=\lceil\frac{g}{2}\rceil=g-\omega(xy)= \widehat{\omega}(x^+y^-)=\widehat{\omega}(x^-y^+).$$
 		Otherwise, by Proposition~\ref{pro:ad in EDC}, $ad_{{\rm EDC}(B,\pi)}(x^+,y^-)=ad_{{\rm EDC}(B,\pi)}(x^-,y^+)=-p-1<0$, so we have   
 			\begin{equation*}
 			\begin{array}{ll}
 				f_{g+1}(ad_{{\rm EDC}(B,\pi)}(x^+,y^-)) &= f_{g+1}(ad_{{\rm EDC}(B,\pi)}(x^-,y^+)) \\
 				&= g+1+(-p-1)=g-p \\
 				&= g-\omega(xy)= \widehat{\omega}(x^+y^-)=\widehat{\omega}(x^-y^+).
 			\end{array} 
 		\end{equation*}
 		
 		The case $ad_{(B,\pi)}(x,y)=p>0$ could be verified similarly, we do not repeat again.
% 		Now, assume $x$ and $y$ are two vertices in $B$ with $ad_{(B,\pi)}(x,y)=p<0$. So $\omega(xy)=g+p$. By Proposition~\ref{pro:ad in EDC}, $ad_{{\rm EDC}(B,\pi)}(x^+,y^-)=ad_{{\rm EDC}(B,\pi)}(x^-,y^+)=-p>0$. Thus $f_{g+1}(ad_{{\rm EDC}(B,\pi)}(x^+,y^-))=f_{g+1}(ad_{{\rm EDC}(B,\pi)}(x^-,y^+))=-p=g-\omega(xy)= \widehat{\omega}(x^+y^-)=\widehat{\omega}(x^-y^+)$. If $d_{B}(x,y)=\lfloor\frac{g}{2}\rfloor$, then $p=-\lfloor\frac{g}{2}\rfloor$, and  $\omega(xy)=g+p=\lceil\frac{g}{2}\rceil$. By Proposition~\ref{pro:ad in EDC}, $ad_{{\rm EDC}(B,\pi)}(x^+,y^+)=ad_{{\rm EDC}(B,\pi)}(x^-,y^-)=\lceil\frac{g}{2}\rceil>0$, so 
% 		$$f_{g+1}(ad_{{\rm EDC}(B,\pi)}(x^+,y^+))=f_{g+1}(ad_{{\rm EDC}(B,\pi)}(x^-,y^-))=\lceil\frac{g}{2}\rceil=\omega(xy)=\widehat{\omega}(x^+y^+)=\widehat{\omega}(x^-y^-).$$ 	
% 		Otherwise, by Proposition~\ref{pro:ad in EDC}, $ad_{{\rm EDC}(B,\pi)}(x^+,y^+)=ad_{{\rm EDC}(B,\pi)}(x^-,y^-)=p-1 <0$. So, we have  
% 		\begin{equation*}
% 			\begin{array}{ll}
% 				f_{g+1}(ad_{{\rm EDC}(B,\pi)}(x^+,y^+)) &= f_{g+1}(ad_{{\rm EDC}(B,\pi)}(x^-,y^-)) \\
% 				&= g+1+(p-1)=g+p \\
% 				&= \omega(xy)= \widehat{\omega}(x^+y^+)=\widehat{\omega}(x^-y^-).
% 			\end{array} 
% 		\end{equation*}
% 	
% 	
% 		Thus we proved our claim.
 	\end{proof}
 	
 	Let $\mathcal{T}'$ be the family of all the triangles in $\widehat{B}$ and $\mathcal{E}'$ consisting of the edges of the triangles in $\mathcal{T}'$, we shall show that $\mathcal{T}'$ is $(g+1)$-closed and nonempty. 
 	
 	Let $p,q,r$ be positive integers such that $(p,q,r) \in \mathcal{L}_{g+1}$ and $e=x^{\alpha}y^{\beta} \in \mathcal{E}'$ with $\widehat{\omega}(e)=p$, where $x,y \in V(B)$, $\alpha,\beta \in \{+,-\}$ (it is possible that $x=y$, in which case, $\alpha \neq \beta$). Following Theorem~\ref{thm:bounds if and only if} and Definition~\ref{def:triangle g closed}, we shall prove that there is a triangle in $\mathcal{T}'$ on $e$. That is equivalent to finding a vertex $z^{\gamma}$, where $z \in V(B)$, $\gamma \in \{+,-\}$, such that either
 	\begin{equation}
 		\label{EDC-traget-1}
 		\tag{C1-1}
 		\widehat{\omega}(z^{\gamma}x^{\alpha})=q,  \ \ \widehat{\omega}(z^{\gamma}y^{\beta})=r,
 	\end{equation} 
 	or 
 	\begin{equation}
 		\label{EDC-traget-2}
 		\tag{C1-2}
 		\widehat{\omega}(z^{\gamma}x^{\alpha})=g+1-q, \ \  \widehat{\omega}(z^{\gamma}y^{\beta})=g+1-r.
 	\end{equation} 
 
 	Note that the statement clearly holds for the case $g=1$, indeed, in this case,  $(B,\pi)$ is a negative loop. Then ${\rm EDC}(B,\pi)$ is a digon, which is obviously $\SP$-complete. For the case that $g=2$,  $(B,\pi)$ is a digon. In the complete ${\rm EDC}(B,\pi)$-distance graph (containing weighted edge $uv$ for any $u,v \in V(B)$), there are four triangles, each of which is a weighted triangle isomorphic to $\bigtriangleup(1,1,2)$, and it is trivial to check that $\mathcal{T'}$ consisting of these four triangles is $g$-closed. 
 	
	Now, consider that $g\geq 3$ and let $e$ be an edge with $\widehat{\omega}(e)=p$. We first assume that $e= x^+x^-$: it follows that $p=g=\max\{|p|,|q|,|r|\}$. By Proposition~\ref{g-good triples}\ref{g-good-product-positive}, $2g \leq g+q+r \leq 2g+2$ and $g+q+r$ is even, we have $q+r \in \{g, g+2\}$.  

	If $q+r=g+2$, then we have $\min\{q,r\} \geq 2$, so $g+1-q, g+1-r \leq g-1$. Moreover, $p+g+1-q+g+1-r=p+g=2g$ is even, thus by Proposition~\ref{g-good triples}\ref{g-good-product-positive}, $(p,g+1-q,g+1-r) \in \mathcal{L}_g$.   By Theorem~ \ref{thm:bounds if and only if}, there is a triangle $xyz \in \mathcal{T}$ such that either $\omega(zx)=q-1$ or $\omega(zx)=g-(q-1)=g+1-q$. Therefore, if $\omega(zx)=q-1$, then $\widehat{\omega}(z^-x^+)=g+1-q$, and $\widehat{\omega}(z^-x^-)=q-1=g+1-r$, so \eqref{EDC-traget-2} holds with $z^\gamma=z^-$. If $\omega(zx)=g+1-q$, then $\widehat{\omega}(z^+x^+)=g+1-q$, and $\widehat{\omega}(z^+x^-)=q-1=g+1-r$, so \eqref{EDC-traget-2} holds with $z^\gamma=z^+$.
 
  	If $q+r = g$, then $(p,q,r) = (g,q,g-q) \in \mathcal{L}_g$ (as $g+q+g-q=2g$ and by Proposition~\ref{g-good triples}\ref{g-good-product-positive}). So there is a triangle $xyz \in \mathcal{T}$ such that either $\omega(zx)=q$ or $\omega(zx)=g-q$. If $\omega(zx)=q$, then $\widehat{\omega}(z^+x^+)=q$, and $\widehat{\omega}(z^+x^-)=g-q=r$, thus \eqref{EDC-traget-1} holds with $z^\gamma=z^+$. If $\omega(zx)=g-q$, then $\widehat{\omega}(z^-x^+)=q$, and $\widehat{\omega}(z^-x^-)=g-q=r$, thus \eqref{EDC-traget-1} holds with $z^\gamma=z^-$.

\medskip
  
	Now, without loss of generality, we assume that $e=x^+y^{\beta}$ and $x\neq y$. It follows that $\omega(xy)=p<g$ if $\beta=+$ and $\omega(xy)=g-p$ if $\beta=-$. We consider the following two cases.

\medskip
\noindent
{\bfseries Case 1.} \emph{$\max\{q,r\} = g$.} 
In this case, without loss of generality, assume that $q=g$. As by Proposition~\ref{g-good triples}\ref{g-good-product-positive} $2g\leq p+g+r \leq 2g+2$ and $p+g+r$ is even, we have $p+r \in \{g, g+2\}$.

If $p+r=g$, then $(p,q,r)=(p,g,g-p)$. Note that $\widehat{\omega}(x^-x^+)=q=g$, and $\widehat{\omega}(x^-y^{\beta})=g-p=r$, by setting $z^{\gamma}=x^-$, \eqref{EDC-traget-1} holds. 

Next, assume that $p+r=g+2$, i.e., $r=g+2-p$. Then, $\min\{p,r\} \geq 2$, and by Proposition~\ref{g-good triples}\ref{g-good-product-positive} we have that   $(p,1,p-1),(g-p,g-1,p-1) \in \mathcal{L}_g$.

\begin{itemize}
	\item If $\beta=+$, then $\omega(xy)=p$. As $(p,1,p-1)\in \mathcal{L}_g$,  there is a triangle $xyz \in \mathcal{T}$ such that either $\omega(zx)=1$ and $\omega(zy)=p-1$, or $\omega(zx)=g-1$ and $\omega(zy)=g-p+1$. For the former case, $\widehat{\omega}(z^+x^+)=1$, and $\widehat{\omega}(z^+y^+)=p-1$; for the latter case, $\widehat{\omega}(z^-x^+)=1$, $\widehat{\omega}(z^-y^+)=p-1$.
	\item If $\beta=-$, then $\omega(xy)=g-p$. As $(g-p,g-1,p-1)\in \mathcal{L}_g$, there exists a triangle $xyz \in \mathcal{T}$ such that either $\omega(zx)=g-1$ and $\omega(zy)=p-1$, or $\omega(zx)=1$ and $\omega(zy)=g+1-p$. For the former case, $\widehat{\omega}(z^-x^+)=1$, and $\widehat{\omega}(z^-y^-)=p-1$. For the latter case, $\widehat{\omega}(z^+x^+)=1$, $\widehat{\omega}(z^+y^-)=p-1$. 
\end{itemize}

For each case above, since $g+1-q=1$ and $g+1-r=p-1$, \eqref{EDC-traget-2} holds.

\medskip
\noindent
{\bfseries Case 2.} \emph{$\max\{q,r\} <g$.} First, assume that $q+r \leq g+1$. Recall that $p < g$  (because the only edges of weight $g$ in $(\widehat{B^*},\widehat{\omega})$ are those of the form $x^+x^-$), so we have $p+q+r \leq 2g$. 
Thus we have $(p,q,r),(-p,-q,r) \in \mathcal{L}_g$. As $T_g(p,q,r)$ is equivalent to $T_g(-p,-q,r)$, $T_g(-p,-q,r)$ is equivalent to $T_g(g-p,g-q,r)$, so $(g-p,g-q,r) \in \mathcal{L}_g$. 

\begin{itemize}
	\item If $\beta=+$, then $\omega(xy)=p$. As $(p,q,r) \in \mathcal{L}_g$, there exists a triangle $xyz \in \mathcal{T}$ such that either $\omega(zx)=q$ and $\omega(zy)=r$, or $\omega(zx)=g-q$ and $\omega(zy)=g-r$. For the former case, $\widehat{\omega}(z^+x^+)=q$, $\widehat{\omega}(z^+y^+)=r$; for the latter case, $\widehat{\omega}(z^-x^+)=q$, $\widehat{\omega}(z^-y^+)=r$. 
	
	\item If $\beta=-$, then $\omega(xy)=g-p$. As $(g-p,g-q,r) \in \mathcal{L}_g$, there exists a triangle $xyz \in \mathcal{T}$ such that either $\omega(zx)=g-q$ and $\omega(zy)=r$, or $\omega(zx)=q$ and $\omega(zy)=g-r$. For the former case, we have that $\widehat{\omega}(z^-x^+)=q$, $\widehat{\omega}(z^-y^-)=r$. For the latter case, we have that $\widehat{\omega}(z^+x^+)=q$, $\widehat{\omega}(z^+y^-)=r$.  
\end{itemize}
For each case above, \eqref{EDC-traget-2} holds.

Now, assume that $q+r \geq g+2$, and let $q'=g+1-q$ and $r'=g+1-r$. Thus, we have that $q'+r' \leq g < g+1$ and $\max\{q',r'\} < g$. Then, by the previous part, that is, when $q+r \leq g+1$, we know that there exists a vertex such that \eqref{EDC-traget-2} holds. More precisely, there exists a vertex $z^\gamma$ such that $\widehat{\omega}(z^{\gamma}x^{\alpha})=g+1-q'=q$, and $\widehat{\omega}(z^{\gamma}y^{\beta})=g+1-r'=r$, that is, \eqref{EDC-traget-1} holds.

This completes the proof.
\end{proof}
 
\section{A class of smaller $g$-wide $\SP$-complete signed graphs}\label{sec-newbound}

In this section, we describe a new family of $\SP$-complete signed graphs from $\mathcal G_{10}\cup \mathcal G_{11}$. For each $g\geq 2$, we will construct a $g$-wide $\SP$-complete signed graph of order $\lfloor g^2/2\rfloor$. These graphs are smaller than the previously known examples: for odd values of $g$, bounds of order $(g-1)^2$ were constructed ~\cite{BFN17}; for even values of $g$, the only previously known examples were the signed projective cubes, of order $2^{g-1}$~\cite{BFN19}.

For any pair of integers $(a,b)$, %with $a > b+1$, 
let ${\rm C}(2a,b)$ denote the Cartesian product $C_{2a}\square P_{b}$, viewed as a cylinder. The graph ${\rm C}(2a,b)$ is of diameter $a+b-1$, and for any vertex $v$ with degree~$3$ in ${\rm C}(2a,b)$, there is a unique vertex at distance $a+b-1$  of $v$ which is therefore called {\em antipodal} of $v$, denoted $\bar{v}$. For any pair of integers $(a,b)$, the {\em Augmented Cylindrical grid}  of dimensions $2a$ and $b$, denoted by ${\rm AC}(2a,b)$, is obtained from ${\rm C}(2a,b)$ by adding an edge between each pairs of  
antipodal vertices, we denote by $J$ the set of edges between antipodal pairs  (see Figure~\ref{AC(8,4)} for an example). More specifically, let ${\rm AC}(2a,b)$ be the graph defined on vertex set $\{0,1,\ldots,2a-1\}\times\{0,1,\ldots,b-1\}$ such that a pair $\{(i_1,j_1),(i_2,j_2)\}$ is an edge if 
\begin{itemize}
	\item $i_1 = i_2$ and $|j_1-j_2| = 1$ (vertical edges), or
	\item $j_1 = j_2$ and $|i_1 - i_2| \in \{1,2a-1\}$ (horizontal edges), or
	\item $|i_2-i_1| + |j_2-j_1| = a+b-1$.  ($J$)
\end{itemize}
\begin{figure}[ht!]
	\centering
	\begin{tikzpicture}	
	\draw[step=0.5,black,thick] (0,0) grid (3.5,1.5);
	\foreach \i in {0,0.5,...,3.5}{
	\foreach \j in {0,0.5,1} {
		\filldraw(\i,\j)circle(0.05);
		\draw[dashed] (0,\j)..controls (0.5,\j+0.3) and  (3,\j+0.3)..(3.5,\j);
	}
	\filldraw(\i,1.5)circle(0.05);
	\draw (0,1.5)..controls (0.5,1.8) and  (3,1.8)..(3.5,1.5);
	}
	\foreach \i in {0,0.5,...,1.5}{
		\draw [red,line width =0.8pt] (\i,0)--(\i+2,1.5);
		\draw [red,line width =0.8pt] (\i,1.5)--(\i+2,0);}
	\end{tikzpicture}
	\caption{The augmented Cylindrical grid $AC(8,4)$. Sloped (red) edges are the edges in $J$.}
	\label{AC(8,4)}
\end{figure}
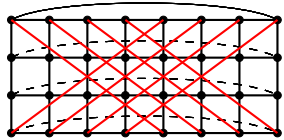

  We note that  the signed graph $({\rm AC}(2a,b),J)$ could be $\mathcal{SP}_g$-complete for various choices of $a$ and $b$ (for example, $P_5 \square P_6$ is $\mathcal{SP}_8$-complete). Here, we will prove this for $a=\lfloor\frac{g}{2}\rfloor$, $b=\lceil\frac{g}{2}\rceil$, which results in a family of very symmetric signed graphs.
 
For convenience and readability, we let   $\T(g)= {\rm C}(2\lfloor\frac{g}{2}\rfloor,\lceil\frac{g}{2}\rceil)$. Also we let $\TT(g)= {\rm AC}(2\lfloor\frac{g}{2}\rfloor,\lceil\frac{g}{2}\rceil)$, in short of \emph{twisted tube of dimension $g$}, in the sense that ${\rm AC}(2\lfloor\frac{g}{2}\rfloor,\lceil\frac{g}{2}\rceil)$ looks like a twisted toroidal grid. Indeed, ${\rm AC}(2\lfloor\frac{g}{2}\rfloor,\lceil\frac{g}{2}\rceil)$ can  be btained from $C(2\lfloor\frac{g}{2}\rfloor,\lceil\frac{g}{2}\rceil+1)$ by identifying each pair of vertices $(i,0)$ and $(i+\lfloor\frac{g}{2}\rfloor,\lceil\frac{g}{2}\rceil+1)$, where $0 \leq i \le \lfloor\frac{g}{2}\rfloor$, in the fashion of a \emph{Dehn twist} studied in algebraic topology.
%\footnote{See this video  for an illustration  \url{https://www.youtube.com/watch?v=O7qsB5boS7s}. In the video, the torus rotates $360^\circ$,  but  in our definition, $180^{\circ}$ is enough.}

\subsection{Properties of the signed graph $(\TT(g),J)$}
 
\begin{lemma}
	\label{property of TT(g)-J}
	For any integer $g\geq 2$, the following statements are true.
	\begin{enumerate}[leftmargin =3em, label=(\roman*)]
		\item \label{TT(g) is a subgraph of PC(g)} $(\TT(g),J)$ is a subgraph  of $\SPC(g-1)$.
		\item \label{TT(g)-transitive} $(\TT(g),J)$ is vertex-transitive.
		\item \label{common negative g-cycle} Any two vertices in $(\TT(g),J)$ belong to a common negative cycle of length $g$.
		\item \label{TTg is g-wide} $(\TT(g),J)$ is $g$-wide.
	\end{enumerate} 
\end{lemma}

\begin{proof}
	\ref{TT(g) is a subgraph of PC(g)} We first label the edges of $(\TT(g),J)$ with canonical vectors  
	$\{e_1,e_2,\ldots,e_{g-1}\}$ of $\{0,1\}^{g-1}$ and $e_J$ (all coordinates are $1$)  as follows (indices are now to be understood modulo $2\lfloor\frac{g}{2}\rfloor$):
	\begin{itemize}
		\item $\{(i-1,j),(i,j)\}$ with label $e_i$ if $i \le \lfloor\frac{g}{2}\rfloor$ and with $e_{i-\lfloor\frac{g}{2}\rfloor}$ otherwise.
		\item $\{(0,j),(2\lfloor\frac{g}{2}\rfloor-1,j)\}$ with label $e_{\lfloor\frac{g}{2}\rfloor}$.	
		\item $\{(i,j-1),(i,j)\}$ with label $e_{\lfloor\frac{g}{2}\rfloor+j}$.
		\item $\{(i,j),(i+\lfloor\frac{g}{2}\rfloor,j+\lfloor\frac{g}{2}\rfloor)\}$ with label $e_J$.  
	\end{itemize} 
	Note that the binary sum of the labels of the edges along any cycle of $\TT(g)$ is the all-zero vector. Conversely, if the sum of the labels along a walk is the all-zero vector, then this walk is closed. Then for any path from vertex $(0,0)$ to some vertex $v$ of $\TT(g)$, the binary sum of the labels is the same. Thus we may define the mapping $\phi$ from the vertices of $(\TT(g),J)$ to the vertices of $\SPC(g-1)$ such that for any vertex $v$ of $(\TT(g),J)$, $\phi(v)$ is the binary sum of the labels along any path from $(0,0)$ to $v$. Observe that the number of different coordinates between $\phi(u)$ and $\phi(v)$ is exactly the same a the number of different coordinates between $u$ and $v$. On the other hand, if $uv$ is a positive edge in $(\TT(G),J)$, then $uv \notin J$, and $u$ and $v$ differ in exactly one coordinate; if $uv$ is a negative edge, then $uv \in J$, and $u$ and $v$ differ in all coordinates.
	So by Proposition~\ref{pro:SPC-coordinate-def}, the mapping $\phi$ is an injective homomorphism from $(\TT(g),J)$ to $\SPC(g-1)$, which implies that $(\TT(g),J)$ is a subgraph of $\SPC(g-1)$.

        \medskip
	\ref{TT(g)-transitive}. Let $v_1=(i_1,j_1)$ and $v_2 = (i_2,j_2)$ be two vertices of $\TT(g)$. Let $\phi$ be the mapping 
 \begin{equation*}
 	\phi:(i,j) \longmapsto \left\{\begin{array}{ll}
 	((i + i_2 - i_1) \bmod 2\gl, j+j_2-j_1), & \text{if} \ 0 \leq j \leq  \gc-1-(j_2-j_1), \\
 	((i + \gl + i_2 - i_1) \bmod 2\gl, (j+j_2-j_1) \bmod \gc), & \text{if} \ j \geq \gc-(j_2-j_1). \\
 	\end{array}
 	\right.
 	\end{equation*}
	First observe that $\phi$ is an automorphism of $\TT(g)$ mapping $v_1$ to $v_2$, as $\phi^{-1}$ exists, one can check it, we omit the details. On the other hand, in the signed graph $\phi(\TT(g),J)$, the negative edges are $\{\{(i,j_2-j_1),(i,j_2-j_1-1)\}|0 \leq i \leq 2\gl-1\}$. So switching at $V=\{(i,j)|j \geq j_2-j_1\}$ in $\phi(TT(g),J)$ gives us $(\TT(g),J)$.
	
	\medskip
        \ref{common negative g-cycle}. Since $\TT(g)$ is vertex-transitive, we may assume that one of these two vertices is the origin $(0,0)$. Let $i$ be an integer between $0$ and $2\gl-1$ and $j$ be an integer between $0$ and $\gc-1$, where $i+j \neq 0$. We need to prove that $(0,0)$ and $(i,j)$ are in a common negative cycle of length $g$. By the symmetries of $\TT(g)$, we may assume that $i \leq \gl$ and $i \leq \gc-1$. If we forget about the antipodal edges, there is a shortest path from $(0, 0)$ to $(\gl, \gc-1)$ going through $(i, j)$. Together with the antipodal edge $\{(0,0),(\gl,\gc-1)\}$, we get a negative cycle of length $g$ in $(\TT(g),J)$.  

        \medskip
	\ref{TTg is g-wide}. By Definition~\ref{def:SPC} and Lemma \ref{lem:girthEDC}, $\SPC(g-1)$ is $g$-wide. By~\ref{TT(g) is a subgraph of PC(g)} and the fact that taking subgraphs will not decrease the girths $g_{ij}$, the statement follows. 
\end{proof}

In the sequel, for any two vertices $u_1=(i_1,j_1)$ and $u_2 = (i_2,j_2)$ in $\T(g)$, we define

 $$d^{+}_{\T(g)}(u_1,u_2) =|i_1-i_2|+|j_1-j_2|, \\ \quad d^{-}_{\T(g)}(u_1,u_2) = 2\gl-|i_1-i_2|+|j_1-j_2|.$$
Intuitively speaking, $d^{+}_{\T(g)}(u_1,u_2)$ is the shortest length of a $(u_1-u_2)$-path that does not use any edges of type $\{(0,j),(2\gl-1,j)\}$, while $d^{-}_{\T(g)}(u_1,u_2)$ does. 
As $$d^{+}_{\T(g)}(u_1,u_2)+d^{-}_{\T(g)}(u_1,u_2) = 2\gl+2|j_1-j_2| \leq 2\gl+2(\gc-1)=2g-2,$$  
we have the following observation.
 
\begin{observation}
	\label{obs:diatance in T(g)}
	For any two vertices $u=(i_1,j_1)$ and $v=(i_2,j_2)$ in $\T(g)$,
	$$d_{{\rm T}(g)}(u,v)= \min\{d^{+}_{\T(g)}(u,v), d^{-}_{\T(g)}(u,v)\} \leq  g-1.$$
	Consequently, if $|i_1-i_2| \leq \gl$, then $d^{+}_{\T(g)}(u,v) \leq d^{-}_{\T(g)}(u,v)$, so $d_{{\rm T}(g)}(u,v)= d^{+}_{\T(g)}(u,v)$. Otherwise,  $d^{+}_{\T(g)}(u,v) > d^{-}_{\T(g)}(u,v)$, so $d_{{\rm T}(g)}(u,v)= d^{-}_{\T(g)}(u,v)$.
\end{observation}
 
 The following observation is easy but useful, and also mentioned in \cite{BFN17,BFN19}.
\begin{observation}
	\label{distance decide in unbalance cycle}
	Let $(G,\sigma)$ be a $g$-wide signed graph and let $C$ be a negative cycle of length $g$ in $(G,\sigma)$. Then, for any pair $(u, v)$ of vertices of $C$, the distance in $G$ between $u$ and $v$ is determined by their distance in $C$.
\end{observation}

By Observation~\ref{obs:diatance in T(g)}, the following special case holds.
 
\begin{observation}
	\label{obs:coordinates of distance t}
	Suppose $u = (i,j)$ is a vertex in $\T(g)$ with $0 \le i < 2\gl$ and $0 \le j < \gc$, and $t$ is an integer satisfying $1 \leq t \leq g-1$. Then, $d_{\T(g)}((0,0),(i,j))=t$ if and only if either  
	\begin{itemize}
		\item $i \leq \gl$ and $i+j=t$, or
		\item $i > \gl$ and $i-j=2\gl-t$.
	\end{itemize}
\end{observation}

\begin{lemma}
	\label{lem:fg(ad)=d}
	For any two vertices $u$ and $v$ in $(\TT(g),J)$, $f_g(ad_{(\TT(g),J)}(u,v))=d_{{\rm T}(g)}(u,v)$.
\end{lemma}

\begin{proof}
	By Lemma~\ref{property of TT(g)-J}\ref{TTg is g-wide}, $\TT(g)$ is $g$-wide, and by Lemma~\ref{property of TT(g)-J}\ref{common negative g-cycle}, for any two vertices $u,v$ in $\TT(g)$, there is a negative cycle of length $g$ in $(\TT(g),J)$ containing $u,v$.  
	
	If $d_{{\rm T}(g)}(u,v) \leq \gl$, then by Observation~\ref{distance decide in unbalance cycle}, $d_{\TT(g)}(u,v)=d_{{\rm T}(g)}(u,v) \leq \gl$, which means that there exists a positive $(u-v)$-path of length $d_{{\rm T}(g)}(u,v)$ in $(\TT(g),J)$. Thus, by Definition~\ref{def-eq-algebraic distance}, $ad_{(\TT(g),J)}(u,v)=d_{{\rm T}(g)}(u,v) > 0$, hence $f_g(ad_{(\TT(g),J)}(u,v))=d_{{\rm T}(g)}(u,v)$; 
	
	If $d_{{\rm T}(g)}(u,v) > \gl$, then by Observation~\ref{distance decide in unbalance cycle}, $d_{\TT(g)}(u,v)=g-d_{{\rm T}(g)}(u,v) < d_{{\rm T}(g)}(u,v)$, which implies that there does not exist a positive $(u-v)$-path of length $d_{{\rm TT}(g)}(u,v)$ in $(\TT(g),J)$. Thus, by Definition~\ref{def-eq-algebraic distance}, $ad_{(\TT(g),J)}(u,v)=d_{{\rm T}(g)}(u,v)-g < 0$, hence $f_g(ad_{(\TT(g),J)}(u,v))=g+ad_{(\TT(g),J)}(u,v)=d_{{\rm T}(g)}(u,v)$.
\end{proof}

\subsection{$(\TT(g),J)$ is $\SP$-complete.}

In this section, we provide a new family of smaller bounds that are $\SP$-complete.

\begin{theorem}
	\label{thm:smallest bound}
	For every integer $g\geq 2$, the signed graph $(\TT(g),J)$, of order $\lfloor g^2/2\rfloor$, is $\SP$-complete.
\end{theorem}

\begin{proof}
	Let $(B,w)$ be a weighted signed graph where $B$ is a complete graph on vertex set $V(\TT(g))$, and for each edge $uv$, $\omega(uv)=f_g(ad_{(\TT(g),J)}(u,v))$. By Lemma~\ref{property of TT(g)-J}~\ref{common negative g-cycle} and Definition~\ref{def-girth transformed}, $(B,w)$ is a $(g-1)$-partial girth-transformed $(TT(g),J)$-distance graph. 	It is clear that the edge set of $(B,w)$ is non-empty. Let $\mathcal{T}$ be the collection of all triangles in $(B,w)$, we shall show that $\mathcal{T}$ is $g$-closed.

	Let $p,q,r$ be three integers satisfying $1 \leq p,q,r \leq g-1$, such that $(p,q,r)$ is $g$-wide, $\mathcal{E}$ be the set of edges appeared in $\mathcal{T}$.  Assume $e=xy \in \mathcal{E}$ with $\omega(e)=p$. We shall find a triangle $xyz \in \mathcal{T}$ such that either $\omega(zx)=q$ and $\omega(zy)=r$, or $\omega(zx)=g-q$ and $\omega(zy)=g-r$. By Lemma~\ref{property of TT(g)-J}\ref{TT(g)-transitive}, $\TT(g)$ is vertex-transitive, so we may assume that $x=(0,0)$. By the horizontal symmetries of $\TT(g)$ (recall that there is an edge $\{(2\gl-1,j),(0,j)\}$ for each $j$), we may assume that $y=(a,b)$, where $0 \leq a \leq \gl$, so by Observation~\ref{obs:coordinates of distance t}, $b=p-a$. Therefore, by Lemma~\ref{lem:fg(ad)=d}, it suffices to show that there exists a vertex $z=(c,d)$ in $\TT(g)$ where $c$ and $d$ are two integers satisfying $0 \leq c \leq 2\gl-1$ and $0 \leq d \leq \gc-1$, such that either	
		\begin{equation}
		\label{smallest-traget-1}
		\tag{C2-1}
		d_{\T(g)}(z,x)=q, \ \ d_{\T(g)}(z,y)=r,
		\end{equation} 
		or 
		\begin{equation}
		\label{smallest-traget-2}
		\tag{C2-2}
		d_{\T(g)}(z,x)=g-q, \ \ d_{\T(g)}(z,y)=g-r.
		\end{equation}   
	We may also assume that $q \leq \gl$, for otherwise, we can replace $q$ with $g-q$ and replace $r$ with $g-r$ such that \eqref{smallest-traget-1} or \eqref{smallest-traget-2} holds. 
	
	By Proposition~\ref{g-good triples}\ref{g-good-product-positive}, $p,q,r$ satisfy the triangle-inequality. So $|a+b-q|=|p-q|\leq  r \leq \min\{p+q,2g-p-q\}=\min\{a+b+q,2g-a-b-q\}$, also  $r$ and $p+q = a+b+q$ have the same parity. 
	We also observe that all of $a+b+q$, $a+|b-q|$, $2\gl+b-a-q$ and $2g-a-b-q$ have the same parity, so depending on the value of $r$, we determine $z=(c,d)$ as follows.
	\begin{enumerate}[label=(\roman*)]
		\item \label{S1} $|a+b-q|\leq r \leq a+|b-q|$.
		\begin{itemize}
			\item   If $b \geq q$, then $r=a+b-q$, and we let $c=0$ and $d=q$.
			\item If $b < q$, then $|a+b-q| \leq r \leq a+q-b$, let $c = \frac{a-b+q-r}{2}$ and $d = \frac{q+r-a+b}{2}$.
		\end{itemize} 
		\item \label{S3} If $a+|b-q|+2 \leq  r \leq \min\{2\gl+b-a-q,a+b+q, 2g-a-b-q\}$, then let  
		$c = 2\gl+\frac{a+b-q-r}{2}$ and $d=\frac{a+b+q-r}{2}$.  
		\item \label{S4} If $2\gl+b-a-q+2 \leq r \leq \min\{a+b+q,2g-a-b-q\}$, then let  
		$c = \frac{a-b-q+r}{2}$ and $d =g-\frac{a-b+q+r}{2}$. 
	\end{enumerate}
 
	As $(p,q,r)$ is $g$-wide, $p+q+r=a+b+q+r$ is even, and consequently, $a+b-q-r$, $a+b+q-r$, $a-b-q+r$ and $a-b+q+r$ are all even, so in each case, $c$ and $d$ are integers. We now proceed to prove the validity of our choices.

\medskip
        
	It is trivial to verify that \eqref{smallest-traget-1} holds for Case~\ref{S1} with $b \geq q$. 

\medskip
        
	For Case~\ref{S1} with $b < q$, note that $c=a-\frac{a+b-q+r}{2}$, as $|a+b-q| \leq r \leq a+q-b$, we have
	\begin{equation}
	\label{last-eq-1-c<a}
		0= \frac{a-b+q-(a+q-b)}{2}  \leq c \leq a-\frac{a+b-q+|a+b-q|}{2} \leq a \leq \gl.
	\end{equation}
	Similarly, recall that $d=b+\frac{r-(a+b-q)}{2}$, so we have
	\begin{equation}
	\label{last-eq-1-b<d}
		b \leq b+\frac{|a+b-q|-(a+b-q)}{2} \leq d \leq \frac{(a+q-b)+(q+b-a)}{2}= q \leq \gl.
	\end{equation}
Observe that $c+d=q$, so by Observation~\ref{obs:coordinates of distance t}, we have 
\begin{equation}
	\label{last-1-distance=q}
	d_{\T(g)}(z,x)=q.
\end{equation}
 By Inequality \eqref{last-eq-1-c<a}, we have $|a-c| \leq \gl$. Thus by Observation~\ref{obs:diatance in T(g)}, Inequalities~\eqref{last-eq-1-c<a}, \eqref{last-eq-1-b<d},  
\begin{equation}
\label{last-1-distance=r}
	d_{\T(g)}(z,y)=d^+_{\T(g)}(z,y)=(a-c)+(d-b)=r.
\end{equation}
Therefore, by  Inequalities~\eqref{last-1-distance=q}, \eqref{last-1-distance=r},   \eqref{smallest-traget-1} holds and the case is done. 

\medskip

For Case~\ref{S3}, note that $c=a+\gl+\frac{2\gl+b-a-q-r}{2}$ and  $r \leq 2\gl+b-a-q$, we have  
$$c \geq a+\gl+\frac{2\gl+b-a-q-(2\gl+b-a-q)}{2}=a+\gl \geq \gl.$$
On the other hand, as $c = 2\gl+\frac{a+b-q-r}{2}$  and $r \geq a+|b-q|+2$, we have
$$c \leq 2\gl+\frac{a+b-q-(a+|b-q|+2)}{2}=2\gl-1+\frac{b-q-|b-q|}{2} \leq  2\gl-1.$$
Thus we have
	\begin{equation}
\label{last-eq-2-a<c}
\gl   \leq \gl +a  \leq  c \leq   2\gl-1,
\end{equation}
Similarly, note that $d=q+\frac{a+b-q-r}{2}$ and $r \geq a+|b-q|+2$, we have
$$d \leq q+ \frac{a+b-q-(a+|b-q|+2)}{2}=q-1+\frac{b-q-|b-q|}{2} \leq q-1.$$
On the other hand, as $r \leq a+b+q$, $d \geq \frac{a+b+q-(a+b+q)}{2}=0$. Therefore,
\begin{equation}
\label{last-eq-2-b<d}
0 \leq d \leq q-1+\frac{b-q-|b-q|}{2} \leq q-1.
\end{equation}

Note that $c-d= 2\gl-q$, by Inequality~\eqref{last-eq-2-a<c} and Observation~\ref{obs:coordinates of distance t}, we have
\begin{equation}
	\label{last-2-distance=q}
	d_{\T(g)}(z,x)=q.
\end{equation}
By Inequality~\eqref{last-eq-2-b<d}, if $b \geq q$, then $d \leq q-1<b$, otherwise $d \leq q-1+(b-q)=b-1$. So it always holds that $d<b$. Recall that Inequality~\eqref{last-eq-2-a<c} ensures that $c-a \geq \gl$, so by Observation~\ref{obs:diatance in T(g)},
\begin{equation}
	\label{last-2-distance=r}
	d_{\T(g)}(z,y) = d^-_{\T(g)}(z,y)=2\gl-(c-a)+(b-d)=r.
\end{equation}
Thus, by Inequality~\eqref{last-2-distance=q} and Inequality~\eqref{last-2-distance=r}, Inequality~\eqref{smallest-traget-1} holds, and this case is done. 

\medskip

For Case~\ref{S4}, recall that $c = \frac{a-b-q+r}{2}$ and $2\gl+b-a-q+2 \leq r \leq a+b+q$, it follows that 
\begin{equation}
	\label{last-eq-3-c<a}
	0 < \gl-q+1 = \frac{a-b-q+(2\gl+b-a-q+2)}{2} \leq c \leq  \frac{a-b-q+(a+b+q)}{2}=a \leq \gl.
\end{equation} 
Similarly, recall that $d =g-\frac{a-b+q+r}{2}=b+g-\frac{a+b+q+r}{2}$, and $r \leq 2g-a-b-q$, we have
$$d \geq  b+g-\frac{a+b+q+(2g-a-b-q)}{2} =b \geq 0.$$
On the other hand, as $r \geq 2\gl+b-a-q+2$, we have 
$$d \leq g-\frac{a-b+q+(2\gl+b-a-q+2)}{2}=\gc-1.$$
Therefore, we have
\begin{equation}
	\label{last-eq-3-b<d}
	0 \leq b  \leq d \leq \gc-1.
\end{equation}

Note that $c+d=g-q$, so by Observation~\ref{obs:coordinates of distance t}, we have 
\begin{equation}
	\label{distance-3=g-q}
	d_{\T(g)}(z,x)=g-q.
\end{equation}
By Inequality~\eqref{last-eq-3-c<a}, $0 \leq a-c \leq \gl$. Thus, by Observation~\ref{obs:diatance in T(g)}, Inequality~\eqref{last-eq-3-c<a} and Inequality~\eqref{last-eq-3-b<d}, 	
\begin{equation}
	\label{distance3=g-r}
	d_{\T(g)}(z,y)=d^+_{\T(g)}(z,y)=(a-c)+(d-b)=g-r.
\end{equation}
Thus, by Inequality~\eqref{distance-3=g-q} and Inequality~\eqref{distance3=g-r}, we know that \eqref{smallest-traget-2} holds, and this case is done.

\medskip
	This completes the proof of this theorem.
\end{proof}

\section{Concluding remarks}

In this work we have observed a strong connection between the notion of Extended Double Cover and some strong conjectures in extension of the Four-Color Theorem. We propose a stronger conjecture as follows:

\begin{conjecture}
Given a signed graph $(B,\pi)$ in $\mathcal{C}_{11}\cup \mathcal{C}_{10}$, if it is $\mathcal{P}$-complete, then ${\rm EDC}(B,\pi)$ is also $\mathcal{P}$-complete.
\end{conjecture}

In support of this conjecture, we showed that the claim holds if we work with the subclass of signed $K_4$-minor-free graphs.

Let $\mathcal{SP}_k$ be the class of signed $K_4$-minor-free graphs $(G,\sigma)$ satisfying $g_{ij}(G,\sigma) \geq g_{ij}(C_{-k})$. For even values of $k$, this is the class of signed bipartite $K_4$-minor free graphs of negative girth at least $k$, and for odd values of $k$, the class of signed antibalanced $K_4$-minor-free graphs of odd-girth $k$. 

We have provided  nearly optimal (in terms of number of vertices of the underlying graph) bound $(B,\pi)$ of order 2$\lfloor \frac{k}{2}\rfloor \lceil \frac{k}{2} \rceil$ for $\mathcal{SP}_k$ satisfying $g_{ij}(B,\pi) \geq g_{ij}(C_{-k})$.

The best possible such bounds are of order 2, 3, 6, 8, 12, 15 for $k=2,3,4,5,6,7$ respectively~\cite{BFN17,BFN19}. This suggest the following formula for the best order of such a bound: $\lfloor \frac{k}{2}\rfloor(\lceil  \frac{k}{2} \rceil+1)$. These formulas support a search among grid-like graphs, and we have provided some nearly optimal graphs based on this class of graphs.

We note that when $k$ is an odd number, the study of the homomorphism properties of $\mathcal{SP}_k$ is the same as the study of homomorphism properties of the class of series-parallel graphs of odd girth at least $k$. Precise bounds on the circular chromatic number in this class are given in~\cite{PZ02}, and on the fractional chromatic number, in~\cite{BFN17,FS15,GX16}. The optimal bounds of order 3, 8, 15 for the cases $k=3,5,7$ from~\cite{BFN17} each have both circular and fractional chromatic numbers that are the same as the best bound for that of series-parallel graphs of odd girth at least $k$ for $k=3,5,7$, hence strengthening results on both the circular chromatic number and the fractional chromatic number. We expect that this will be the case for general odd values of $k$. This is an alternative motivation for finding the optimal choice for $B$.

\bibliographystyle{abbrv}
\bibliography{ReferenceFile}

\end{document}